\documentclass[review]{elsarticle}
\usepackage{amsmath}
\usepackage{graphicx,setspace}
\usepackage{amsfonts,amsmath, amssymb, amsthm}
\usepackage{mathrsfs}
\usepackage{epstopdf}
\usepackage{float}
\usepackage{lineno,hyperref}
\usepackage{color}
\usepackage{subfigure}
\usepackage{bm}
\usepackage{graphicx}
\usepackage{amssymb, amsthm}
\usepackage{mathrsfs}
\usepackage{epstopdf}
\usepackage{float}
\usepackage{lineno,hyperref}
\usepackage{color}
\usepackage{subfigure}
\usepackage{cases}
\usepackage{geometry}
\usepackage{makecell}
\numberwithin{equation}{section}

\journal{}

\begin{document}

\newtheorem{definition}{Definition}
\newtheorem{lemma}{Lemma}
\newtheorem{remark}{Remark}
\newtheorem{theorem}{Theorem}
\newtheorem{proposition}{Proposition}
\newtheorem{assumption}{Assumption}
\newtheorem{example*}{Example}
\newtheorem{corollary}{Corollary}
\def\e{\varepsilon}
\def\Rn{\mathbb{R}^{n}}
\def\Rm{\mathbb{R}^{m}}
\def\E{\mathbb{E}}
\def\hte{\bar\theta}
\def\cC{{\mathcal C}}
\newcommand{\Cp}{\mathbb C}
\newcommand{\R}{\mathbb R}
\renewcommand{\L}{\mathcal L}
\newcommand{\supp}{\mathrm{supp}\,}

\numberwithin{equation}{section}

\begin{frontmatter}

\title{{\bf The stochastic fractional nonlinear Schr\"odinger equations in $H^{\alpha}$ and structure-preserving algorithm
}}
\author{\normalsize{\bf Ao Zhang$^{a}\footnote{ aozhang1993@csu.edu.cn}$
Yanjie Zhang$^{b}\footnote{zhangyj2022@zzu.edu.cn}$
Pengde Wang$^{c}\footnote{pengde\_wang@yeah.net}$
Xiao Wang$^{d}\footnote{xwang@vip.henu.edu.cn}$
Jinqiao Duan$^{e}\footnote{duan@gbu.edu.cn}$} \\[10pt]
\footnotesize{${}^a$ School of Mathematics and Statistics and HNP-LAMA,  Central South University,  Changsha 410083, China} \\[5pt]
\footnotesize{${}^b$ School of Mathematics and Statistics, Zhengzhou University, Zhengzhou 450052,  China}  \\[5pt]
\footnotesize{${}^c$ College of Mathematics and Information Science, Henan University of Economics and Law, Zhengzhou 450046,  China} \\[5pt]
\footnotesize{${}^d$ School of Mathematics and Statistics, Henan University, Kaifeng 475001, China}\\[5pt]
\footnotesize{${}^e$  Department of Mathematics and Department of Physics, Great Bay University, Dongguan, Guangdong 523000, China}
}

\begin{abstract}
In this paper, we first investigate the global existence of a solution for the stochastic fractional nonlinear Schr\"odinger equation with radially
symmetric initial data in a suitable energy space $H^{\alpha}$. We then show that the stochastic fractional nonlinear Schr\"odinger equation in the Stratonovich sense forms an infinite-dimensional stochastic Hamiltonian system, with its phase flow preserving symplecticity. Finally, we develop a stochastic midpoint scheme for the stochastic fractional nonlinear Schr\"odinger equation from the perspective of symplectic geometry. It is proved that the stochastic midpoint scheme satisfies the corresponding symplectic law in the discrete sense. A numerical example is conducted to validate the efficiency of the theory.
\end{abstract}

\begin{keyword}
 Global well-posedness,  Hamiltonian systems, stochastic midpoint scheme, stochastic multi-symplectic structure.

\end{keyword}

\end{frontmatter}


\section{Introduction.}

The fractional nonlinear Schr\"odinger equation, as an extension of the standard Schr\"odinger equation, arises in diverse fields like nonlinear optics \cite{Longhi15}, quantum physics \cite{BM23}, and water propagation \cite{Pusateri14}. Inspired by the Feynman path approach to quantum mechanics, Laskin \cite{Laskin02} utilized the path integral over L\'evy-like quantum mechanical paths to derive a fractional Schr\"odinger equation. Kirkpatrick \cite{Kirkpatrick02} examined a broad range of discrete nonlinear Schr\"odinger equations on the lattice $h\mathbb{Z}$ with mesh size $h>0$. They demonstrated that the resulting dynamics converge to a nonlinear fractional Schr\"odinger equation as $h\rightarrow 0$. Ionescu and Pusateri \cite{IP14} demonstrated the global existence of small, smooth, and localized solutions for a specific fractional semilinear cubic nonlinear Schr\"odinger equation in one dimension. Shang and Zhang \cite{SZ} investigated the existence and multiplicity of solutions for the critical fractional Schr\"odinger equation. They demonstrated the presence of a nonnegative ground state solution and explored the relationship between the number of solutions and the set's topology. Choi and Aceves \cite{CA23} proved that the solutions to the discrete nonlinear Schr\"odinger equation with non-local algebraically decaying coupling converged strongly in $L^2(\mathbb{R}^n)$ to those of the continuum fractional nonlinear Schr\"odinger equation. Frank and his colleagues \cite{LS15} established general uniqueness findings for radial solutions of linear and nonlinear equations involving the fractional Laplacian $(-\Delta)^{\alpha}$ with $\alpha \in(0,1)$ for any space dimensions $n\geq 1$. Wang and Huang \cite{Huang15} presented an energy-conserving difference scheme for the nonlinear fractional Schr\"odinger equations and provided a thorough analysis of the conservation property. Duo and Zhang \cite{DZ16} proposed three mass-conservative Fourier spectral methods for solving the fractional nonlinear Schr\"odinger equation.

In certain situations, it is necessary to consider randomness. Understanding the impact of noise on wave propagation is a major challenge, as it can significantly alter the qualitative behavior and lead to new properties. Stochastic nonlinear Schr\"odinger equations are employed as nonlocal models for wave propagation in various physical applications. For the stochastic Schr\"odinger equation driven by Gaussian noise, Bouard and Debussche \cite{Debussche99, Debussche05} examined a conservative stochastic nonlinear Schr\"odinger equation and the impact of multiplicative Gaussian noises, demonstrating the global presence and uniqueness of solutions. Herr et al. \cite{rockner19} examined the scattering behavior of global solutions to stochastic nonlinear Schr\"odinger equations with linear multiplicative noise. Barue et al. \cite{Zhang17} demonstrated that the explosion could be effectively prevented over the entire time interval $[0, \infty)$. Deng et al. \cite{Deng22} researched the spread of randomness in nonlinear dispersive equations using the theory of random tensors. Bouard and Debusschev \cite{Deb02, Deb04} developed a semi-discrete method for numerically approximating the stochastic nonlinear Schr\"odinger equation in the Stratonovich sense. Cui and his colleagues \cite {H19} gave the strong convergence rate of the splitting scheme for stochastic Schr\"odinger equations. Liu \cite{JL130, JL13} examined the error of a semi-discrete method and demonstrated that the numerical scheme possessed a strong order of $\frac{1}{2}$. Chen and Hong \cite{CH16} introduced a broad category of stochastic symplectic Runge-Kutta methods for the temporal direction of the stochastic Schr\"odinger equation in the Stratonovich sense. Anton and Cohen \cite{AC18} demonstrated the robust convergence of an exponential integrator for the time discretization of stochastic Schr\"odinger equations influenced by additive or multiplicative It\^o noise. Cui et al.\cite{CW17, CH18, CHL17} introduced the stochastic symplectic and multi-symplectic methods and demonstrated their convergence with a temporal order of one in probability. They also proved a strong rate $1/2$ as well as a weak rate $1$ of a splitting scheme for a damped stochastic cubic Schr\"odinger equation with linear multiplicative trace-class noise and large enough damping term. Br\'ehier and Cohen \cite{CDC} investigated the qualitative characteristics and convergence order of a splitting scheme for the stochastic Schr\"odinger equations with additive noise. They showed the stochastic symplecticity of the numerical solution and consistently preserved the expected mass. Yang and Chen \cite{Chen17} demonstrated the presence of martingale solutions for the stochastic fractional nonlinear Schr\"odinger equation on a limited interval. Zhang et al. \cite{ZZD} proved the existence and uniqueness of a global solution to the stochastic fractional nonlinear Schr\"odinger equation in $L^{2}(\mathbb{R}^n)$. For the stochastic Schr\"odinger equation driven by jump noise, Liu and his collaborators \cite{Liu21} established a new version of the stochastic Strichartz estimate for the stochastic convolution driven by jump noise. Few studies involve the well-posedness of the solution in the energy space $H^{\alpha}$, its stochastic symplectic and convergence theorem, and symplectic semi-discretization for stochastic fractional nonlinear Schr\"odinger equation with multiplicative noise.

This paper focuses on the examination of the stochastic fractional nonlinear Schr\"odinger equation (SFNSE) with multiplicative noise in $H^\alpha(\mathbb{R}^n)$, i.e.,
\begin{equation}
\label{orie07}
\left\{
\begin{aligned}
&i du-\left[(-\Delta)^{\alpha} u+\lambda |u|^{2\sigma}u\right] dt=u \circ dW(t), \quad x \in \mathbb{R}^n, \quad t \geq 0, \\
&u(0)=u_0,
\end{aligned}
\right.
\end{equation}
where $u$ is a complex-valued process defined on $\mathbb{R}^+ \times \mathbb{R}^{n}$, $0<\sigma <\infty$, and $W(t)$ is a Wiener process. The parameter $\lambda = 1$ (resp. $\lambda=-1$) corresponds to the defocusing (resp. focusing) case. The fractional Laplacian operator $(-\Delta)^{\alpha}$ with an admissible exponent $\alpha \in (\frac{1}{2},1)$ is involved. The notation $\circ$ stands for the Stratonovitch integral.

Consider a probability space $(\Omega, \mathcal{F}, \mathbb{P})$, a filtration $\left(\mathcal{F}_{t}\right)_{t \geq 0}$, and a sequence of independent Brownian motions $\left(\beta_{k}\right)_{k \in \mathbb{N}^+}$ associated with this filtration. Given $\left(e_{k}\right)_{k \in \mathbb{N}^+}$ an orthonormal basis of $L^{2}\left(\mathbb{R}^{n}, \mathbb{R}\right)$, and a linear operator $\Phi$ on $L^{2}\left(\mathbb{R}^{n}, \mathbb{R}\right)$ with a real-valued kernel $k$:
\begin{equation}
\Phi h(x)=\int_{\mathbb{R}^{n}} k(x, y) h(y) d y, \quad h \in L^{2}\left(\mathbb{R}^{n}, \mathbb{R}\right).
\end{equation}
Then the process
\begin{equation*}
W(t, x, \omega):=\sum_{k=1}^{\infty} \beta_{k}(t, \omega) \Phi e_{k}(x), \quad t \geq 0, \quad x \in \mathbb{R}^{n}, \quad \omega \in \Omega,
\end{equation*}
is a Wiener process on $L^{2}\left(\mathbb{R}^{n}, \mathbb{R}\right)$ with covariance operator $\Phi \Phi^{*}$, and  the equation \eqref{orie07} can be rewritten as
\begin{equation}
i d u-\left[(-\Delta)^{\alpha} u+\lambda |u|^{2\sigma}u \right] d t=u d W-\frac{1}{2} i u F_{\Phi} d t,
\end{equation}
where the function $F_{\Phi}$ is given by
\begin{equation}
F_{\Phi}(x)=\sum_{k=1}^{\infty}\left(\Phi e_{k}(x)\right)^{2}, \quad x \in \mathbb{R}^{n}.
\end{equation}

We are interested in the global existence and uniqueness of the solution in a suitable energy space $H^{\alpha}$ and its stochastic symplecticity.  In contrast to the case of stochastic nonlinear Schr\"odinger equation in $L^2(\mathbb{R}^{n})$. There are three essential difficulties in our problems. The first difficulty is that the Sobolev embedding $H^{\alpha_2}(\mathbb{R}^{n})\hookrightarrow H^{\alpha_1}(\mathbb{R}^{n})$ with $\alpha_2>\alpha_1$ is not compact in $\mathbb{R}^n$. The second difficulty is that the fractional operator $(-\Delta)^{\alpha}$ is non-local and hence deriving the uniform boundedness of the energy $H[u]$ is much more involved than the standard Laplacian operator under the assumption condition $ \sum_{k=1}^{\infty}\int_{\mathbb{R}^n}\frac{\left(\Phi e_k(x)-\Phi e_k(y)\right)^2}{|x-y|^{n+2\alpha}}dx \leq \mathcal{R}$. The third difficulty is that we need to provide an ``appropriate" decomposition of the fractional Laplace operator and prove the preservation of symplecticity for the exact solution.

This paper is organized as follows. In Section 2, we introduce notations for the coefficients required by our main result and some definitions. Section 3 is dedicated to the existence of a global solution of equation \eqref{orie07} in the energy space $H^{\alpha}$. In Section 4, we demonstrate that the stochastic fractional nonlinear Schr\"odinger equation is an infinite-dimensional Hamiltonian system and analyze its stochastic multi-symplectic structure. A numerical test is conducted to validate the efficiency. Section 5 will be reserved for some concluding remarks. Some tedious proofs of lemmas are left in Appendix A.

\section{Preliminaries}
Throughout this paper, we use some notations. The capital letter $C$ represents a positive constant, with its value possibly changing from one line to the next. The notation $C_p$ is specifically used to indicate that the constant depends on the parameter $p$. For $p\geq 2$, $L^{p}$ denotes the Lebesgue space of complex-valued functions. The inner product in $L^{2}\left(\mathbb{R}^{n}\right)$ is defined as
\begin{equation*}
(f, g)={\bf{Re}} \int_{\mathbb{R}^{n}} f(x) \bar{g}(x) dx, \quad \text{for} ~~f, g \in L^{2}\left(\mathbb{R}^{n}\right).
\end{equation*}

Given a Banach space $B$, we denote by $R(H; B)$ the $\gamma$-radonifying operator from $H$ into $B$ \cite{NN7}, equipped with the norm
\begin{equation*}
\|K\|^2_{R(H,B)}=\widetilde{\mathbb{E}}\left|\sum_{k=1}^{\infty}\gamma_kKe_k\right|^2_{B},
\end{equation*}
where $(e_k)_{k\in \mathbb{N}}$ is any orthonormal basis of $H$, and $(\gamma_k)_{k\in \mathbb{N}}$ is any sequence of independent normal real-valued random variables on a probability space $\left(\widetilde{\Omega}, \widetilde{\mathcal{F}}, \widetilde{P} \right)$. $\widetilde{E}$ is the expectation on $\left(\widetilde{\Omega}, \widetilde{\mathcal{F}}, \widetilde{P} \right)$.

We define the standard space $H^{\alpha}\left(\mathbb{R}^n\right)$ for tempered distributions $u \in$ $\mathcal{S}^{\prime}\left(\mathbb{R}^n\right)$ with Fourier transform $\hat{u}$ satisfying $\left(1+|\xi|^2\right)^{{\alpha}/ 2} \hat{u} \in L^2\left(\mathbb{R}^n\right)$. We may use the abbreviated notation $L^r\left(0, T ; L_x^p\right)$ for $L^r\left([0, T] ; L^p\left(\mathbb{R}^n\right)\right)$, or even $L_T^r\left(L_x^p\right)$ when the interval $[0, T]$ is specified and fixed. Given two separable Hilbert spaces $H$ and $\widetilde{H}$, the notation $L_{HS}(H,\widetilde{H})$ denotes the space of Hilbert-Schmidt operators from $H$ into $\widetilde{H}$. Let $\Phi: H\rightarrow \widetilde{H} $ be a bounded linear operator. The operator $\Phi$ is called the Hilbert-Schmidt operator if there is an orthonormal basis $(e_{k})_{k\in \mathbb{N}^+}$ in $H$ such that
\begin{equation*}
\|\Phi\|^2_{L_{HS}(H,\widetilde{H})}=\operatorname{tr} \Phi^* \Phi=\sum_{k=1}^{\infty}\left|\Phi e_k\right|_{\widetilde{H}}^2< \infty.
\end{equation*}
When $H =\widetilde{H}=L^2(\mathbb{R}^n; \mathbb{R})$, $L_{HS}\left(L^2(\mathbb{R}^n; \mathbb{R}), L^2(\mathbb{R}^n; \mathbb{R})\right)$ is simply denoted by $L_{HS}$. Note that if $K$ is a Hilbert space such that $K \subset E$ with a continuous embedding then $L_{HS}(H, K) \subset R(H, E)$ with a continuous embedding.

We also use $\nabla$ to denote the gradient operator in Euclidean space. Let $L(E; F)$ denote the space of bounded linear operators from $E$ into $F$. Let $(m,p)$ be in $\mathbb{N}\times [1,\infty)$. The Banach space $W^{m,p}(\Omega)$ consists of measurable functions $u$ on $\Omega$ such that $\nabla^{\beta}u\in L^p(\Omega)$ in the sense of distributions, for every multi-index $\beta$ with $|\beta| \leq m$. $W^{m,p}(\Omega)$ is equipped with the norm
\begin{equation*}
\|u\|_{W^{m,p}}=\sum_{|\beta|\leq m}\|\nabla^{\beta}u\|_{L^p}.
\end{equation*}
The space $\mathcal{C}^{1,1}_{loc}\cap \mathcal{L}_{\alpha}$ consists of measurable functions $u$ on $\mathbb{R}^n$ such that $u\in L^{1}_{loc}$ and $\int_{\mathbb{R}^n}\frac{|u(x)|}{1+|x|^{n+2\alpha}}dx < \infty$.

In the following, we present criteria for determining if a stochastic system is a stochastic Hamiltonian system, which comes from \cite{GM02, WP19}.
\begin{definition}
For the $2n$-dimensional stochastic differential equation in Stratonovich sense
\begin{equation}
\left\{
\begin{aligned}
dp&=f(t,p,q)dt+\sum_{k=1}^{r}\sigma_k(t,p,q)\circ dW_k(t), ~~p(t_0)=p_0, \\
dq&=g(t,p,q)dt+\sum_{k=1}^{r}\gamma_k(t,p,q)\circ dW_k(t), ~~q(t_0)=q_0, \\
\end{aligned}
\right.
\end{equation}
if there are functions $H(t,p,q)$ and $H_k(t,p,q)(k=1,2, \cdots, r)$ such that
\begin{equation}
\begin{aligned}
f(t,p,q)&=-\frac{\partial H(t,p,q)}{\partial q}^{T}, ~~~\sigma_k(t,p,q)=-\frac{\partial \widetilde{H}_k(t,p,q)}{\partial q}^{T}, \\
g(t,p,q)&=\frac{\partial H(t,p,q)}{\partial p}^{T}, ~~~\gamma_k(t,p,q)=\frac{\partial \widetilde{H}_k(t,p,q)}{\partial p}^{T}, \\
\end{aligned}
\end{equation}
for $k=1, \cdots, r$, then it is a stochastic Hamiltonian system.
\end{definition}
Here, we will discuss fractional derivatives \cite{Duan15}. The fractional Laplace $(-\Delta)^{\alpha}$ is defined as
\begin{equation*}
(-\Delta)^{\alpha}u(x)=C(n,\alpha)~ \mathrm{P.V.}\int_{\mathbb{R}^n}\frac{u(x)-u(y)}{|x-y|^{n+2\alpha}}dy,\\
\end{equation*}
where $\mathrm{P.V.}$ denotes the principal value of the integral, and $C(n,\alpha)$ is a positive constant given by
\begin{equation*}
\begin{aligned}
C(n,\alpha)=\frac{2^{2\alpha}\alpha\Gamma(\frac{n+2\alpha}{2})}{\pi^{\frac{n}{2}}\Gamma(1-\alpha)}.
\end{aligned}
\end{equation*}
For every $\alpha\in(0,1)$, denote by
\begin{equation*}
\begin{aligned}
H^{\alpha}(\mathbb{R}^n)=\left\{u\in L^2(\mathbb{R}^n): \int_{\mathbb{R}^n}\int_{\mathbb{R}^n}\frac{|u(x)-u(y)|^2}{|x-y|^{n+2\alpha}}dxdy < \infty \right\}.
\end{aligned}
\end{equation*}
Then $H^{\alpha}(\mathbb{R}^n)$ is a Hilbert space with inner product given by
\begin{equation*}
(u,v)_{H^{\alpha}(\mathbb{R}^n)}=(u,v)_{L^{2}(\mathbb{R}^n)}+\mathbf{Re} \int_{\mathbb{R}^n}\int_{\mathbb{R}^n}\frac{\left(u(x)-u(y)\right)\left(\bar{v}(x)-\bar{v}(y)\right)}{|x-y|^{n+2\alpha}}dxdy,\quad u,v\in H^{\alpha}(\mathbb{R}^n).
\end{equation*}
In the following, we will also give the notation
\begin{equation*}
(u,v)_{\dot{H}^{\alpha}(\mathbb{R}^n)}=\mathbf{Re}\int_{\mathbb{R}^n}\int_{\mathbb{R}^n}\frac{\left(u(x)-u(y)\right)\left(\bar{v}(x)-\bar{v}(y)\right)}{|x-y|^{n+2\alpha}}dxdy,\quad u, v\in H^{\alpha}(\mathbb{R}^n),
\end{equation*}
and
\begin{equation*}
\|u\|^2_{\dot{H}^{\alpha}(\mathbb{R}^n)}=\int_{\mathbb{R}^n}\int_{\mathbb{R}^n}\frac{\left|u(x)-u(y)\right|^2}{|x-y|^{n+2\alpha}}dxdy,\quad u\in H^{\alpha}(\mathbb{R}^n).
\end{equation*}
Then we have
\begin{equation*}
(u,v)_{H^{\alpha}(\mathbb{R}^n)}=(u,v)_{L^2(\mathbb{R}^n)}+(u,v)_{\dot{H}^{\alpha}(\mathbb{R}^n)},\quad u, v\in H^{\alpha}(\mathbb{R}^n),
\end{equation*}
and
\begin{equation*}
\|u\|^2_{H^{\alpha}(\mathbb{R}^n)}=\|u\|^2_{L^2(\mathbb{R}^n)}+\|u\|^2_{\dot{H}^{\alpha}(\mathbb{R}^n)},\quad u\in H^{\alpha}(\mathbb{R}^n).
\end{equation*}
By \cite{EG12}, we know that
\begin{equation*}
\|(-\Delta)^{\frac{\alpha}{2}}u\|^{2}_{L^{2}(\mathbb{R}^n)}=\frac{1}{2}C(n,\alpha)\|u\|^{2}_{\dot{H}^{\alpha}(\mathbb{R}^n)},
\end{equation*}
and the norm $\|u\|_{H^{\alpha}\left(\mathbb{R}^n\right)}$ is equivalent to the norm $\left(\|u\|_{L^2\left(\mathbb{R}^n\right)}^2+\left\|(-\Delta)^{\frac{\alpha}{2}} u\right\|_{L^2\left(\mathbb{R}^n\right)}^2\right)^{\frac{1}{2}}$ for any $u \in H^{\alpha}\left(\mathbb{R}^n\right)$.

\section{Global Existence}
This section focuses on studying the global existence of the stochastic fractional nonlinear Schr\"odinger equation \eqref{orie07} with radially symmetric initial data in the energy space $H^{\alpha}$. Let us first recall the definition of an admissible pair. We say a pair $(p, q)$ satisfies the fractional admissible condition if
\begin{equation}
  p \in[2, \infty], \quad q \in[2, \infty), \quad(p, q) \neq\left(2, \frac{4 n-2}{2 n-3}\right), \quad \frac{2\alpha}{p}+\frac{n}{q}=\frac{n}{2}.
\end{equation}

The unitary group $S(t):=e^{-i t(-\Delta)^{\alpha}}$ enjoys several types of Strichartz estimates, for instance, non-radial Strichartz estimates, radial Strichartz estimates, and weighted Strichartz estimates. We only recall here radial Strichartz estimates (see, e.g., Ref. \cite{VDD18}).
\begin{lemma}
\label{0estimate}
(Radial Strichartz estimates) For $n \geq 2$ and $\frac{n}{2 n-1} \leq \alpha<1$, there exists a positive constant $C$ such that the following estimates hold:
  \begin{equation}\label{Strichartz}
  \begin{aligned}
  \left\|e^{-i t(-\Delta)^{\alpha}}\psi\right\|_{L^p\left(\mathbb{R}, L^q\right)} &\leq C\|\psi\|_{L^2},\\
  \left\|\int_{0}^{t} S(t-s) f(s) d s\right\|_{L^{p}\left(\mathbb{R}, L^{q}\right)} & \leq C \|f\|_{L^{\beta^{\prime}}\left(\mathbb{R}, L^{\gamma^{\prime}}\right)},
  \end{aligned}
  \end{equation}
  where $\psi$ and $f$ are radially symmetric, $(p, q)$ and $(\beta, \gamma)$ satisfy the fractional admissible condition, and $\frac{1}{\beta^{\prime}}+\frac{1}{\beta}=\frac{1}{\gamma^{\prime}}+\frac{1}{\gamma}=1$.
\end{lemma}

\begin{assumption}\label{assump1}(On the noise)
We assume that $\Phi: L^2 \rightarrow H^{\alpha}$ is a Hilbert-Schmidt operator, i.e.
\begin{equation*}
\|\Phi\|_{L_{H S}(L^2,  H^{\alpha})}:=\left(\sum_{k=1}^{\infty}\left\|\Phi e_k\right\|_{ H^{\alpha}}^2\right)^{1 / 2}<\infty .
\end{equation*}

This implies that
\begin{equation*}
\|\Phi\|_{L_{H S}(L^2;  H^{\alpha})}^2=\sum_{k=1}^{\infty}\left\|\Phi e_k\right\|^2+\sum_{k=1}^{\infty}\left\|(-\Delta)^{\frac{\alpha}{2}} \Phi e_k\right\|^2.
\end{equation*}
\end{assumption}

\begin{assumption}\label{assump2}(On the nonlinearity)
\begin{enumerate}
 \item [\textcolor{black}{$\bullet$}] If $\lambda=-1$ (focusing), let $0 \leq \sigma<\frac{2\alpha}{n}$.
 \item [\textcolor{black}{$\bullet$}] If $\lambda=1$ (defocusing), let $\begin{cases}0 \leq \sigma<\frac{2\alpha}{n-2\alpha}, & \text { for } n \geq 2, \\ \sigma \geq 0, & \text { for } n < 2.\end{cases}$
 \end{enumerate}
\end{assumption}

 We recall the following fractional chain rule \cite{CW91},  which is needed in the local well-posedness for \eqref{orie07}.
\begin{lemma}
\label{lemm1}
Let $F \in C^1(\mathbb{C}, \mathbb{C})$ and $\alpha \in(0,1)$. Then for $1<q \leq q_2<\infty$ and $1<q_1 \leq \infty$ satisfying $\frac{1}{q}=\frac{1}{q_1}+\frac{1}{q_2}$, there exists a positive constant $C$ such that
\begin{equation}
\left\||\nabla|^{\alpha}F(u)\right\|_{L^q} \leq C\left\|F^{\prime}(u)\right\|_{L^{q_1}}\left\||\nabla|^{\alpha}u\right\|_{L^{q_2}} .
\end{equation}
\end{lemma}

Because the stochastic integral $\int_0^{t} S(t-s)(u(s) d W(s))$ is not a local martingale, we can not directly use Burkholder's inequality. We need the following lemma, which is a straightforward generalization of Lemma 3.1 and Lemma 3.2 in \cite{Debussche99}.
\begin{lemma}\label{lemm2}
 Assume that a pair $(p, q)$ satisfies the fractional admissible condition, $q<\frac{2 n}{n-\alpha}$, $T_0>0$, and take $\rho \geq p$. Under Assumption \ref{assump1}, for any $u$ radially adapted process in $L^\rho\left(\Omega ; L^p\left(0, T_0 ; W^{\alpha, q}\right)\right) \cap L^\rho\left(\Omega ; L^{\infty}\left(0, T_0 ; H^{\alpha}\right)\right)$ if $Iu$ is defined for $t_0, t \in\left[0, T_0\right]$ by
\begin{equation*}
I u\left(t_0, t\right)=\int_0^{t_0} S(t-s)(u(s) d W(s)),
\end{equation*}
where $S(t):=e^{-i t(-\Delta)^{\alpha}}$.
Then for any stopping time $\tau$ with $\tau<T_0$ almost surely,
\begin{equation*}
\mathbb{E}\left(\sup _{0 \leq t_0 \leq \tau}\left|I u\left(t_0, \cdot\right)\right|_{L^p\left(0, \tau ; W^{\alpha, q}\right)}^\rho\right) \leq C\left(\rho, T_0,\|\Phi\|_{L_{H S}(L^2,  H^{\alpha})}\right) T_0^\delta \mathbb{E}\left(|u|_{L^{\infty}\left(0, \tau ; H^{\alpha}\right)}^\rho\right)
\end{equation*}
with $\delta>0$.
Furthermore, if we set $J u(t)=I u(t, t)$ then $J u$ has a modification in $L^\rho\left(\Omega ; L^p\left(0, T_0 ; W^{\alpha, q}\right) \cap C\left(\left[0, T_0\right] ; H^{\alpha}\right)\right)$ and
\begin{equation*}
\mathbb{E}\left(|J u|_{L^p\left(0, \tau ; W^{\alpha, q}\right)}^\rho\right) \leq C\left(\rho, T_0,\|\Phi\|_{L_{H S}(L^2,  H^{\alpha})}\right) T_0^\delta \mathbb{E}\left(|u|_{L^{\infty}\left(0, \tau ; H^{\alpha}\right)}^\rho\right)
\end{equation*}
and
\begin{equation*}
\mathbb{E}\left(\sup _{0 \leq t \leq \tau}|J u(t)|_{H^{\alpha}}^\rho\right) \leq C\left(\rho, T_0,\|\Phi\|_{L_{H S}(L^2,  H^{\alpha})}\right) T_0^\delta \mathbb{E}\left(|u|_{L^p\left(0, \tau ; W^{\alpha, q}\right)}^\rho\right).
\end{equation*}
\end{lemma}

Next, we give the local well-posedness of the stochastic fractional nonlinear Schr\"odinger equation \eqref{orie07} with radially symmetric initial data in the energy space $H^{\alpha}$.
\begin{theorem}(Radial local theory).
\label{theorem1}
 Let $n \geq 2$ and $\alpha \in\left[\frac{n}{2 n-1}, 1\right)$ and $0<\sigma<\frac{2 \alpha}{n-2 \alpha}$. Let
\begin{equation*}
p=\frac{4 \alpha(\sigma+2)}{\sigma(n-2 \alpha)}, \quad q=\frac{n(\sigma+2)}{n+\sigma \alpha}. \end{equation*}
Then under Assumption \ref{assump1}, for any $u_0 \in H^{\alpha}$ radial, there exists a stopping time $\tau^*\left(u_0, \omega\right)$ and a unique solution to \eqref{orie07} starting from $u_0$, which is almost surely in $C\left([0, \tau] ; H^{\alpha}\right) \cap L^p\left(0, \tau ; W^{\alpha, q}\right)$ for any $\tau<\tau^*\left(u_0\right)$. Moreover, we have almost surely,
\begin{equation*}
\tau^*\left(u_0, \omega\right)=+\infty \quad \text { or } \quad \limsup _{t\to\tau^*\left(u_0, \omega\right)}|u(t)|_{H^{\alpha}}=+\infty.
\end{equation*}
\end{theorem}

\begin{proof}
It is easy to check that $(p, q)$ satisfies the fractional admissible condition. We choose $(a, b)$ so that
\begin{equation*}
\frac{1}{p^{\prime}}=\frac{1}{p}+\frac{2\sigma}{a}, \quad \frac{1}{q^{\prime}}=\frac{1}{q}+\frac{2\sigma}{b} .
\end{equation*}
We see that
\begin{equation*}
\frac{2\sigma}{a}-\frac{2\sigma}{p}=1-\frac{\sigma(n-2 \alpha)(\sigma+1)}{2 \alpha(\sigma+2)}=: \theta>0, \quad q \leq b=\frac{n q}{n-\alpha q},\quad q<\frac{2 n}{n-\alpha}.
\end{equation*}
The latter fact gives the Sobolev embedding $W^{\alpha, q} \hookrightarrow L^b$. Define the spaces:
\begin{equation*}
X:=\left\{C\left(I, H^\alpha\right) \cap L^p\left(I, W^{\alpha, q}\right):\|u\|_{L^{\infty}\left(I, H^\alpha\right)}+\|u\|_{L^p\left(I, W^{\alpha, q}\right)} \leq M\right\},
\end{equation*}
where $I=[0, \zeta]$ and $M, \zeta>0$ to be chosen later. Then the proof is based on a fixed point argument. Let us briefly discuss the proof, and more details are similar to the Theorem 4.1 in \cite{AD03}.  By Duhamel's formula, it suffices to prove that the functional
\begin{equation*}
\begin{aligned}
  \mathcal{T}(u)(t):= & S(t) u_0-i \lambda \int_0^t S(t-s)|u(s)|^{2 \sigma} u(s)ds \\
& +i \int_0^t S(t-s)(u(s) d W(s))-\frac{1}{2} \int_0^t S(t-s)\left(u(s) F_\Phi\right) d s
\end{aligned}
\end{equation*}
is a contraction in $L^{\rho}\left(\Omega ; X\right)$.

Let $u_1, u_2 \in L^{\rho}\left(\Omega ; X\right)$, then using Strichartz estimates \eqref{Strichartz}, we have almost surely
\begin{equation*}
\begin{aligned}
\left|\mathcal{T} u_1-\mathcal{T}u_2\right|_{X} \leq
&C\left|\left(|u_1|^{2 \sigma} u_1-|u_2|^{2 \sigma} u_2\right)\right|_{L^{p^{\prime}}\left(I, W^{\alpha, q^{\prime}}\right)}+\left|\int_0^t S(t-s)\left(u_1(s)-u_2(s)\right) d W(s)\right|_{X}+C\left|\left(u_1-u_2\right) F_\Phi\right|_{L^{p^{\prime}}\left(I, W^{\alpha, q^{\prime}}\right)} \\
:= & H_1+H_2+H_3.
\end{aligned}
\end{equation*}
The fractional chain rule given in Lemma \ref{lemm1} and H\"older's inequality give
\begin{equation*}
\begin{aligned}
\left\||u|^{2\sigma} u\right\|_{L^{p^{\prime}}\left(I, W^{\alpha, q^{\prime}}\right)} & \leq C\|u\|_{L^a\left(I, L^b\right)}^{2\sigma}\|u\|_{L^p\left(I, W^{\alpha, q}\right)}
  \leq C|I|^\theta\|u\|_{L^p\left(I, L^b\right)}^{2\sigma}\|u\|_{L^p\left(I, W^{\alpha, q}\right)}
& \leq C|I|^\theta\|u\|_{L^p\left(I, W^{\alpha, q}\right)}^{{2\sigma}+1}.
\end{aligned}
\end{equation*}
Similarly,
\begin{equation*}
\begin{aligned}
H_1=\left\||u_1|^{2\sigma}-|u_2|^{2\sigma} v\right\|_{L^{p^{\prime}}\left(I, W^{\alpha, q^{\prime}}\right)} & \leq C\left(\|u_1\|_{L^a\left(I, L^b\right)}^{2\sigma}+\|u_2\|_{L^a\left(I, L^b\right)}^{2\sigma}\right)\|u_1-u_2\|_{L^p\left(I, W^{\alpha, q}\right)} \\
& \leq C|I|^\theta\left(\|u_1\|_{L^p\left(I, W^{\alpha, q}\right)}^{2\sigma}+\|u_2\|_{L^p\left(I, W^{\alpha, q}\right)}^{2\sigma}\right)\|u_1-u_2\|_{L^p\left(I, W^{\alpha, q}\right)}.
\end{aligned}
\end{equation*}
The last term is easily estimated thanks to H\"older's inequality, i.e.,
\begin{equation*}
\begin{aligned}
H_3& \leq C \|I\|^{1-\frac{2}{p}}\left|u_1-u_2\right|_{L^p\left(I, W^{\alpha, q}\right)}\left|F_\Phi\right|_{L^{\frac{n p}{4\alpha}}}
 \leq C \|I\|^{1-\frac{2}{p}}\|\Phi\|_{L_{H S}(L^2,  H^{\alpha})}\left|u_1-u_2\right|_X.
\end{aligned}
\end{equation*}

Thus, using Lemma \ref{lemm2} easily shows that $\mathcal{T}$ is a contraction mapping in $L^{\rho}\left(\Omega, X\right)$ provided $\zeta$ is chosen sufficiently small, depending on $M$.
\end{proof}

\begin{remark}
For $d=1, 0<\sigma<\infty$, the local well-posedness of the stochastic fractional nonlinear Schr\"odinger equation \eqref{orie07} also holds, due to $H^{\alpha}({\mathbb{R}})\hookrightarrow L^{\infty}(\mathbb{R})$.
\end{remark}

The fractional nonlinearity Schr\"odinger shares  the similarity with the classical nonlinear Schr\"odinger equation, which has the formal law for the mass and energy by
\begin{equation*}
M[u]=\int_{\mathbb{R}^n} |u|^{2}dx,
\end{equation*}
and
\begin{equation}\label{Energy}
H[u]=\frac{1}{2}\int_{\mathbb{R}^n} \left|(-\Delta)^{\frac{\alpha}{2}}u\right|^2dx +\frac{\lambda}{2\sigma +2}\int_{\mathbb{R}^n} |u|^{2\sigma+2}dx.
\end{equation}
In the following, we will give the results about the mass $M[u]$ and the energy $H[u]$, which comes from \cite[Proposition 3]{ZZD}.
\begin{proposition}
\label{energy}
Assume that $u_0 \in H^{\alpha}$ radial, Assumption \ref{assump1} holds, and
\begin{equation*}
\sum_{k=1}^{\infty}\int_{\mathbb{R}^n}\frac{\left(\Phi e_k(x)-\Phi e_k(y)\right)^2}{|x-y|^{n+2\alpha}}dx \leq \mathcal{R}
 \end{equation*}
 with a positive constant $\mathcal{R}$. For any stopping time $\tau<\tau^*(u_0)$, we have
\begin{equation}
M(u(\tau))=M(u_0),\quad a.s.,
\end{equation}
and
\begin{equation}\label{0energy1}
\begin{aligned}
H(u(\tau))&=H(u_0)+\mathbf{Im} \int^{\tau}_0 \int_{\mathbb{R}^n} (-\Delta)^{\alpha}\bar{u}udxdW
-\frac{1}{2}\mathbf{Re}\int^{\tau}_0 \int_{\mathbb{R}^n}\left[(-\Delta)^{\alpha}\bar{u}\right]uF_{\Phi}dxdt\\
&\quad+\frac{1}{2}\sum_{k=1}^{\infty} \int^{\tau}_0 \int_{\mathbb{R}^n} (-\Delta)^{\alpha}\left(\bar{u}\Phi e_k(x)\right){u}\Phi e_k(x)dxdt.
\end{aligned}
\end{equation}
\end{proposition}

\begin{theorem}(Global well-posedness).
\label{theorem2}
 Assume that $u_0 \in H^{\alpha}$ radial, Assumptions \ref{assump1} and \ref{assump2} hold, and
 \begin{equation}\label{Phicond}
\sum_{k=1}^{\infty}\int_{\mathbb{R}^n}\frac{\left(\Phi e_k(x)-\Phi e_k(y)\right)^2}{|x-y|^{n+2\alpha}}dx \leq \mathcal{R}
 \end{equation}
 with a positive constant $\mathcal{R}$. Then there exists a unique global solution of equation \eqref{orie07} in $H^{\alpha}$, i.e., $\tau^{\star}(u_0)=\infty$.
\end{theorem}

\begin{proof}
We only need to get the uniform boundedness of $\|u\|_{H^{\alpha}}$ to ensure the global existence of the solution. So we first obtain the uniform boundedness of the energy $H[u]$. Then the energy evolution of $u$ implies that for any $T_0>0$, any stopping time $\tau<\inf(T_0, \tau^*(u_0))$ and any time $t\leq \tau$,
\begin{equation*}
\begin{aligned}
\mathbb{E}\left[\sup_{t\leq \tau}H(u(t))\right]&\leq H(u_0)+\mathbb{E}\left[\sup_{t\leq \tau} \left|\int^{t}_{0} \int_{\mathbb{R}^n}(-\Delta)^{\alpha}\bar{u}udxdW\right|\right]\\
&\quad+\frac{1}{2}\mathbb{E}\left[\sup_{t\leq \tau} \left|\sum_{k=1}^{\infty} \int^{t}_0 \int_{\mathbb{R}^n} (-\Delta)^{\alpha}\left(\bar{u}\Phi e_k(x)\right){u}\Phi e_k(x)dxds-\mathbf{Re}\int^{t}_0 \int_{\mathbb{R}^n}\left[(-\Delta)^{\alpha}\bar{u}\right]uF_{\Phi}dxds\right|\right]\\
&:=H(u_0)+\mathcal{J}_1+\mathcal{J}_2.
\end{aligned}
\end{equation*}

For the term $\mathcal{J}_1$, one has
\begin{equation}
  \begin{aligned}
    \mathbf{Im}\left[\int^{t}_{0} \int_{\mathbb{R}^n}(-\Delta)^{\alpha}\bar{u}udxdW\right]&=\mathbf{Im}\left[\sum_{{k=1}}^{\infty} \int^{t}_{0} \int_{\mathbb{R}^n}(-\Delta)^{\frac{\alpha}{2}}\bar{u}\cdot(-\Delta)^{\frac{\alpha}{2}}(u\Phi e_k)dxd\beta_k\right]\\
  & =\frac{C(n, \alpha)}{2}\sum_{{k=1}}^{\infty}\mathbf{Im}\left[\int^{t}_{0}\int_{\mathbb{R}^n} \int_{\mathbb{R}^n} \frac{\left(\bar{u}(x)-\bar{u}(y)\right)\left(u(x)\Phi e_k(x)-u(y)\Phi e_k(y)\right)}{|x-y|^{n+2 \alpha}} dydxd\beta_k\right] \\
  &=\frac{C(n, \alpha)}{2}\sum_{{k=1}}^{\infty}\mathbf{Im}\left[\int^{t}_{0}\int_{\mathbb{R}^n} \int_{\mathbb{R}^n} \frac{\Phi e_k(x)\left|u(x)-u(y)\right|^2}{|x-y|^{n+2 \alpha}} dydxd\beta_k\right] \\
  &\quad+\frac{C(n, \alpha)}{2}\sum_{{k=1}}^{\infty}\mathbf{Im}\left[\int^{t}_{0}\int_{\mathbb{R}^n} \int_{\mathbb{R}^n} \frac{\left(\bar{u}(x)-\bar{u}(y)\right)\left(\Phi e_k(x)-\Phi e_k(y)\right)u(y)}{|x-y|^{n+2 \alpha}} dydxd\beta_k\right] \\
  &=\frac{C(n, \alpha)}{2}\sum_{{k=1}}^{\infty}\mathbf{Im}\left[\int^{t}_{0}\int_{\mathbb{R}^n} \int_{\mathbb{R}^n} \frac{\left(\bar{u}(x)-\bar{u}(y)\right)\left(\Phi e_k(x)-\Phi e_k(y)\right)u(y)}{|x-y|^{n+2 \alpha}} dydxd\beta_k\right].\\
  \end{aligned}
  \end{equation}
By using Burkholder-Davis-Gundy inequality, H\"older equality, Young inequality, and  Lemma 3.4 (see \cite{ZD16}), we have
\begin{equation}\label{J1}
  \begin{aligned}
  \mathcal{J}_1&=\mathbb{E}\left[\sup_{t\leq \tau} \left|\int^{t}_{0} \int_{\mathbb{R}^n}(-\Delta)^{\alpha}\bar{u}udxdW\right|\right]\\
  &\leq C\mathbb{E}\left[\left(\int^{\tau}_{0}\left(\sum_{{k=1}}^{\infty}\int_{\mathbb{R}^n} \int_{\mathbb{R}^n} \frac{\left(\bar{u}(x)-\bar{u}(y)\right)\left(\Phi e_k(x)-\Phi e_k(y)\right)u(y)}{|x-y|^{n+2 \alpha}} dxdy\right)^2\mathrm{d}t\right)^\frac{1}{2}\right] \\
  &\leq C\mathbb{E}\left[\left(\int^{\tau}_{0}\left(\sum_{{k=1}}^{\infty}\int_{\mathbb{R}^n} \left(\int_{\mathbb{R}^n} \frac{|u(y)|\left|u(x)-u(y)\right|\left|\Phi e_k(x)-\Phi e_k(y)\right|}{|x-y|^{n+2 \alpha}} dx\right) dy\right)^2dt\right)^\frac{1}{2}\right] \\
  &\leq C\mathbb{E}\left[\left(\int^{\tau}_{0}\left(\sum_{{k=1}}^{\infty}\|u\|^2\int_{\mathbb{R}^n} \left(\int_{\mathbb{R}^n} \frac{\left|\left(u(x)-u(y)\right)\left(\Phi e_k(x)-\Phi e_k(y)\right)\right|}{|x-y|^{n+2 \alpha}} dx\right)^2 dy\right)dt\right)^\frac{1}{2}\right] \\
  &\leq C(M(u_0))\mathbb{E}\left[\left(\int^{\tau}_{0}\left(\sum_{{k=1}}^{\infty}\int_{\mathbb{R}^n} \left(\int_{\mathbb{R}^n} \frac{\left|u(x)-u(y)\right|^2}{|x-y|^{n+2 \alpha}} dx \int_{\mathbb{R}^n} \frac{\left|\Phi e_k(x)-\Phi e_k(y)\right|^2}{|x-y|^{n+2 \alpha}} dx\right) dy\right)dt\right)^\frac{1}{2}\right]\\
  &\leq C(n,\alpha,T_0,M(u_0), \mathcal{R})\mathbb{E}\left[\left(\int^{\tau}_{0}\|(-\Delta)^\frac{\alpha}{2}u\|^2dt\right)^\frac{1}{2}\right]\\
  &\leq C(n,\alpha,T_0,M(u_0), \mathcal{R})+C(n,\alpha,T_0,M(u_0), \mathcal{R})\mathbb{E}\left[\left(\int^{\tau}_{0}\|(-\Delta)^\frac{\alpha}{2}u\|^4dt\right)^\frac{1}{2}\right]\\
  &\leq C(n,\alpha,T_0,M(u_0), \mathcal{R})+C(n,\alpha,T_0,M(u_0), \mathcal{R})\varepsilon\mathbb{E}\left[\sup_{t\leq \tau}\|(-\Delta)^\frac{\alpha}{2}u\|^2\right]\\
  &\quad+C(n,\alpha,T_0,M(u_0), \mathcal{R}, \varepsilon)\int^{\tau}_{0}\mathbb{E}\left(\sup_{r\leq s} \|(-\Delta)^\frac{\alpha}{2}u(r)\|^2\right)ds,
  \end{aligned}
  \end{equation}
where $\varepsilon$ is a positive number.

For the term $\mathcal{J}_2$, we have
  \begin{equation}
  \begin{aligned}
&\sum_{k=1}^{\infty} \int_{\mathbb{R}^n} (-\Delta)^{\alpha}\left(\bar{u}\Phi e_k(x)\right){u}\Phi e_k(x)dxdt-\mathbf{Re} \int_{\mathbb{R}^n}\left[(-\Delta)^{\alpha}\bar{u}\right]uF_{\Phi}dxdt\\
 &=\sum_{k=1}^{\infty}\mathbf{Re}\left\{\left((-\Delta)^{\frac{\alpha}{2}}\left(\bar{u}\Phi e_k\right), (-\Delta)^{\frac{\alpha}{2}}\left(\bar{u}\Phi e_k\right)\right)-\left((-\Delta)^{\frac{\alpha}{2}}\bar{u}, (-\Delta)^{\frac{\alpha}{2}}\left(\bar{u}(\Phi e_k)^2\right)\right)
 \right\}\\
 &=\frac{C(n, \alpha)}{2}\sum_{k=1}^{\infty}\int_{\mathbb{R}^n}\int_{\mathbb{R}^n}\frac{\left(\bar{u}(x)\Phi e_k(x)-\bar{u}(y)\Phi e_k(y)\right)\left(u(x)\Phi e_k(x)-u(y)\Phi e_k(y)\right)}{|x-y|^{n+2\alpha}}dxdy\\
&\quad-\frac{C(n, \alpha)}{2}\sum_{k=1}^{\infty}\mathbf{Re}\int_{\mathbb{R}^n}\int_{\mathbb{R}^n}\frac{\left(\bar{u}(x)-\bar{u}(y)\right)\left(u(x)(\Phi e_k)^2(x)-u(y)(\Phi e_k)^2(y)\right)}{|x-y|^{n+2\alpha}}dxdy\\
 &=\frac{C(n, \alpha)}{2}\sum_{k=1}^{\infty}\mathbf{Re}\int_{\mathbb{R}^n}\int_{\mathbb{R}^n}
 \frac{u(x)\bar{u}(y)\Phi e_k(x)\left(\Phi e_k(x)-\Phi e_k(y)\right)-\bar{u}(x)u(y)\Phi e_k(y)\left(\Phi e_k(x)-\Phi e_k(y)\right)}{|x-y|^{n+2\alpha}}dxdy\\
 &=\frac{C(n, \alpha)}{2}\sum_{k=1}^{\infty}\mathbf{Re}\int_{\mathbb{R}^n}\int_{\mathbb{R}^n}\frac{u(x)\bar{u}(y)\left(\Phi e_k(x)-\Phi e_k(y)\right)\left(\Phi e_k(x)-\Phi e_k(y)\right)}{|x-y|^{n+2\alpha}}dxdy\\
 &\leq\frac{C(n, \alpha)}{2}\sum_{k=1}^{\infty}\int_{\mathbb{R}^n}\int_{\mathbb{R}^n}\frac{|u(x)\left(\Phi e_k(x)-\Phi e_k(y)\right)|\cdot |u(y)\left(\Phi e_k(x)-\Phi e_k(y)\right)|}{|x-y|^{n+2\alpha}}dxdy\\
 &\leq\frac{C(n, \alpha)}{4}\sum_{k=1}^{\infty}\int_{\mathbb{R}^n}\int_{\mathbb{R}^n}\frac{|u(x)|^2\left(\Phi e_k(x)-\Phi e_k(y)\right)^2}{|x-y|^{n+2\alpha}}dxdy+\frac{C(n, \alpha)}{4}\sum_{k=0}^{\infty}\int_{\mathbb{R}^n}\int_{\mathbb{R}^n}\frac{|u(y)|^2\left(\Phi e_k(x)-\Phi e_k(y)\right)^2}{|x-y|^{n+2\alpha}}dxdy\\
 &=\frac{C(n, \alpha)}{4}\sum_{k=1}^{\infty}\left(\int_{\mathbb{R}^n}|u(x)|^2\int_{\mathbb{R}^n}\frac{\left(\Phi e_k(x)-\Phi e_k(y)\right)^2}{|x-y|^{n+2\alpha}}dy\right)dx+\frac{C(n, \alpha)}{4}
 \sum_{k=1}^{\infty}\left(\int_{\mathbb{R}^n}|u(y)|^2\int_{\mathbb{R}^n}\frac{\left(\Phi e_k(x)-\Phi e_k(y)\right)^2}{|x-y|^{n+2\alpha}}dx\right)dy\\
 &\leq \frac{C(n, \alpha)}{2}M(u_0)\mathcal{R}.
 \end{aligned}
 \end{equation}
 Thus we have
 \begin{equation}
 \label{J2}
 \mathcal{J}_2\leq C(n,\alpha,T_0,M(u_0),\mathcal{R}).
 \end{equation}
Thus combined the above inequalities \eqref{J1} and \eqref{J2}, we have
\begin{equation}
\label{1energy1}
  \begin{aligned}
    \mathbb{E}\left[\sup_{t\leq \tau}H(u(t))\right]&\leq H(u_0)+ C(n,\alpha,T_0,M(u_0), \mathcal{R})+C(n,\alpha,T_0,M(u_0), \mathcal{R})\varepsilon\mathbb{E}\left[\sup_{t\leq \tau}\|(-\Delta)^\frac{\alpha}{2}u\|^2\right]\\
  &\quad+C(n,\alpha,T_0,M(u_0), \mathcal{R}, \varepsilon)\int^{\tau}_{0}\mathbb{E}\left(\sup_{r\leq s} \|(-\Delta)^\frac{\alpha}{2}u(r)\|^2\right)ds.
  \end{aligned}
\end{equation}

We first consider the case where $\lambda=1$ and $\sigma<\frac{2\alpha}{n-2\alpha}$, then using the equality \eqref{Energy}, we obtain
\begin{equation}
\begin{aligned}
\mathbb{E}\sup_{t\leq \tau}\|u\|^2_{\dot{H}^{\alpha}}&\leq 2\mathbb{E}\sup_{t\leq \tau} H(u)\\
&\leq 2H(u_0)+ C(n,\alpha,T_0,M(u_0), \mathcal{R})+C(n,\alpha,T_0,M(u_0), \mathcal{R})\varepsilon\mathbb{E}\left[\sup_{t\leq \tau}\|(-\Delta)^\frac{\alpha}{2}u\|^2\right]\\
  &\quad+C(n,\alpha,T_0,M(u_0), \mathcal{R}, \varepsilon)\int^{\tau}_{0}\mathbb{E}\left(\sup_{r\leq s} \|(-\Delta)^\frac{\alpha}{2}u(r)\|^2\right)ds.
\end{aligned}
\end{equation}
Then, using the Gr\"onwall inequality and the conservation of mass, one can deduce
\begin{equation*}
\mathbb{E}\sup_{t\leq \tau}\|u\|^2_{H^{\alpha}}\leq C(n,\alpha,T_0,M(u_0), H(u_0),\mathcal{R}).
\end{equation*}

Treating the case where $\lambda=-1$ and $\sigma<\frac{2\alpha}{n}$, we can utilize Gagliardo--Nirenberg's inequality (see \cite{VDD19}) and Young's inequality to derive the following equation:
\begin{equation}
\label{in6}
\begin{aligned}
\|u\|^{2\sigma+2}_{L^{2\sigma+2}}&\leq C \|u\|^{\frac{n\sigma}{\alpha}}_{\dot{H}^{\alpha}}\cdot \|u\|^{2\sigma+2-\frac{n\sigma}{\alpha}}_{L^2}
\leq C \varepsilon \|u\|^{2}_{\dot{H}^{\alpha}}+ C_{\varepsilon}\|u\|^{2+\frac{4\sigma}{2\alpha-n\sigma}}_{L^2}.
\end{aligned}
\end{equation}
It's important to note that in the last inequality, it is crucial that $\sigma<\frac{2\alpha}{n}$. Then by the equality \eqref{Energy} and inequality \eqref{in6}, we know
\begin{equation*}
\begin{aligned}
\mathbb{E}\sup_{t\leq \tau}\|u\|^2_{\dot{H}^{\alpha}}&=2\mathbb{E}[\sup_{t\leq \tau} H(u)]+\frac{1}{\sigma+1}\mathbb{E}\left[\sup_{t\leq \tau}\|u\|^{2\sigma+2}_{L^{2\sigma+2}}\right]\\
&\leq 2H(u_0)+ C(n,\alpha,T_0,M(u_0), \mathcal{R})+C(n,\alpha,T_0,M(u_0), \mathcal{R})\varepsilon\mathbb{E}\left[\sup_{t\leq \tau}\|(-\Delta)^\frac{\alpha}{2}u\|^2\right]\\
  &\quad+C(n,\alpha,T_0,M(u_0), \mathcal{R}, \varepsilon)\int^{\tau}_{0}\mathbb{E}\left(\sup_{r\leq s} \|(-\Delta)^\frac{\alpha}{2}u(r)\|^2\right)ds.
\end{aligned}
\end{equation*}
By using Gr\"onwall inequality, we obtain
\begin{equation*}
\mathbb{E}\sup_{t\leq \tau}\|u\|^2_{\dot{H}^{\alpha}} \leq C( n,\alpha,T_0, M(u_0), H(u_0),\mathcal{R}).
\end{equation*}
These a priori estimates, combined with local well-posedness, imply the global existence of a unique solution.
\end{proof}

\begin{remark}
The condition \eqref{Phicond} is technical and reasonable. Here we introduce a smooth function $\rho(x)$ defined for $0 \leq x<\infty$ such that $0 \leq \rho(x) \leq 1$, for $x\geq 0$ and
\begin{equation}
\label{pho}
\rho(x)=\begin{cases}0, & \text { if }~ 0 \leq x\leq \frac{1}{2}, \\ 1, & \text { if }~ x \geq 1.\end{cases}
\end{equation}
Taking $\Phi e_k(x)=\rho(\frac{|x|}{k})$ and $\alpha \in (\frac{1}{2}, 1)$, there exists a positive constant $L_1$ independent of $k$ and $y\in \mathbb{R}^n$  such that
\begin{equation*}
\sum_{k=1}^{\infty}\int_{\mathbb{R}^n} \frac{\left|\Phi e_k(x)-\Phi e_k(y)\right|^2}{|x-y|^{n+2\alpha}} d x \leq \sum_{k=1}^{\infty}\frac{L_1}{k^{2\alpha}}< \infty.
\end{equation*}
\end{remark}

\section{Stochastic symplecticity}
  As we know, the deterministic fractional nonlinear Schr\"odinger equation is an infinite-dimensional Hamiltonian system, which characterizes the geometric invariants of the phase flow. In the following, we analyze its stochastic multi-symplectic structure for stochastic fractional nonlinear Schr\"odinger equation \eqref{orie07}.

Assume that $u=p+iq$, then we can rewrite the SFNSE \eqref{orie07} as a  pairs of real-valued equations
\begin{equation}
\label{seperate}
\left\{
\begin{aligned}
dp-(-\Delta)^{\alpha}qdt-\lambda\left(p^2+q^2\right)^{\sigma}qdt-q\circ dW=0,\\
dq+(-\Delta)^{\alpha}pdt+\lambda  \left(p^2+q^2\right)^{\sigma}pdt+p\circ dW=0. \\
\end{aligned}
\right.
\end{equation}
Similar to the classical equation, the SFNSE also has the following results.
\begin{lemma}
The system \eqref{seperate} has the following mass and energy laws.
\begin{equation}
\label{decop}
\begin{aligned}
Mass: M&=\int_{\mathbb{R}^n}(p^2+q^2)dx,\\
Energy: H&=
\frac{1}{2}\int_{\mathbb{R}^n} \left|(-\Delta)^{\frac{\alpha}{2}}p\right|^2dx+\frac{1}{2}\int_{\mathbb{R}^n} \left|(-\Delta)^{\frac{\alpha}{2}}q\right|^2dx  +\frac{\lambda}{2\sigma +2}\int_{\mathbb{R}^n} (p^2+q^2)^{\sigma+1}dx.
\\
\end{aligned}
\end{equation}
\end{lemma}
The following lemma gives the nonlocal character of fractional Laplacian based on its definition.
\begin{lemma}
Let $u\in \mathcal{C}^{1,1}_{loc}\bigcap \mathcal{L}_{\alpha}$ be a function. Then we have
\begin{equation}
-(-\Delta)^{\alpha}u=\mathcal{G}^{2}u:=\mathcal{G}(\mathcal{G}u),
\end{equation}
where $\mathcal{G}u(x, t)=\frac{1}{(2\pi)^{n/2}}\int_{\mathbb{R}^{n}}\left(i|\xi|^{\alpha-1}\xi\right)\hat{u}(\xi,t)e^{ix \cdot \xi}d\xi$.
\end{lemma}
\begin{proof}
By the definition of the fractional Laplacian, we have
\begin{equation*}
\begin{aligned}
-(-\Delta)^{\alpha}u(x,t)&=\frac{1}{(2\pi)^n} \int_{\mathbb{R}^{n}} \left(-|\xi|^{2\alpha}\hat{u}(\xi, t)\right)e^{ix\cdot \xi}d\xi
=\frac{1}{(2\pi)^{n}}\int_{\mathbb{R}^{n}}(i|\xi|^{\alpha-1}\xi)(i|\xi|^{\alpha-1}\xi)\cdot \hat{u}(\xi,t)e^{ix\cdot \xi}d\xi\\
&=:\mathcal{G}\left(\mathcal{G}u\right).
\end{aligned}
\end{equation*}
\end{proof}

\subsection{Stochastic symplectic structure}
To examine the stochastic multi-symplectic structure of the equation \eqref{orie07}, we initially utilize the following decomposition derived from the definition, as presented in \cite[Lemma 2.3]{WH18}.
\begin{lemma}
Let $u$ be a period function. Then we have
\begin{equation}
-(-\Delta)^{\alpha}u=\tilde{\mathcal{G}}^2u:=\tilde{\mathcal{G}}(\tilde{\mathcal{G}}u),
\end{equation}
where $\tilde{\mathcal{G}}$ defined by $\tilde{\mathcal{G}}u:=\sum_{l\in \mathbb{Z}}(iv_l|v_l|^{\alpha-1})\hat{u}_le^{iv_l(x-a)}$ is a skew-adjoint operator.
\end{lemma}

The lemma below demonstrates that the system \eqref{seperate} exhibits the standard property.
\begin{lemma}
The system \eqref{seperate} is a stochastic Hamiltonian system
\begin{equation}
\left\{
\begin{aligned}
dp&=-\nabla_{p}Hdt+q\circ dW(t), ~~p(t_0)=p_0, \\
dq&=\nabla_{q}Hdt-p\circ dW_k(t), ~~q(t_0)=q_0, \\
\end{aligned}
\right.
\end{equation}
with  the Hamiltonian function
\begin{equation}
H(t,p,q)=-\frac{1}{2}\left(p^{T}\mathcal{G}(\mathcal{G}p)+q^{T}\mathcal{G}(\mathcal{G}q)\right)+\frac{\lambda}{2(\sigma+1)}(p^2+q^2)^{\sigma+1}, ~~ \widetilde{H}(t,p,q)=-\frac{1}{2}(p^2+q^2).
\end{equation}
\end{lemma}

One of the inherent canonical properties of the Hamiltonian system is the symplectic or multi-symplectic of its flows, for instance \cite{DC06, DC08} and references therein.  These particular integrators have thus naturally come into the realm of a stochastic partial differential equation, for example \cite{GNM02, CAA14, KB12} and reference therein. The stochastic Schr\"odinger equation can be interpreted as a canonical infinite-dimensional Hamiltonian system \cite{JH19}.

Let $v=\mathcal{G}p$, $\omega=\mathcal{G}q$, $z=(p,q,v,\omega)^{T}$, then the equation \eqref{seperate} can be rewritten as
\begin{equation}
\label{symple}
\mathbf{M}z_t+\mathbf{K}(\mathcal{G}z)=\nabla_{z} S_1(z) +\nabla_{z} S_2(z)\circ \dot{{W}}_t,
\end{equation}
where
$$
\mathbf{M}=
\begin{pmatrix}
0 & 1 & 0 & 0 \\
-1 & 0 & 0 & 0 \\
0 & 0 & 0 & 0 \\
0 & 0 & 0 & 0 \\
\end{pmatrix}
,~~~~
\mathbf{K}=
\begin{pmatrix}
0 & 0 & -1 & 0 \\
0 & 0 & 0 & -1 \\
1 & 0 & 0 & 0 \\
0 & 1 & 0 & 0 \\
\end{pmatrix}
,~~~~
$$
and
\begin{equation*}
S_1(z)=\frac{1}{2}v^2+\frac{1}{2}\omega^2-\frac{\lambda}{2(\sigma+1)}(p^2+q^2)^{\sigma+1}, ~~~~S_2(z)=-\frac{1}{2}(p^2+q^2).
\end{equation*}

Next, we define $\Psi$ and $\partial_{x}\kappa$ by
\begin{equation}
\Psi=\frac{1}{2}dz\wedge \mathbf{M}dz, ~~\partial_{x}\kappa={\frac{1}{2}}dz\wedge \mathbf{K}d\left(\mathcal{G}z\right).
\end{equation}

Then we have the following theorem.
\begin{theorem}(Stochastic  multi-symplectic conservation law).
\label{msy}
The stochastic multi-symplectic Hamiltonian system  \eqref{symple} preserves the stochastic  multi-symplectic conservation law locally, i.e.,
\begin{equation}
\label{conservation}
d_t\Psi +\partial_{x}\kappa(x,t)=0, ~~a.s.
\end{equation}
In other words,
\begin{equation*}
\Psi(t_1,x)-\Psi(t_0,x)=-\int^{t_1}_{t_0}\partial_{x}\kappa(t,x)dt.
\end{equation*}
\end{theorem}
\begin{proof}

Differentiating \eqref{symple} gives the following variational equation
\begin{equation}
\label{variat}
\mathbf{M}dz_t+\mathbf{K}d(\mathcal{G}z)=\bigtriangledown_{zz}S_1(z)dz+\bigtriangledown_{zz}S_2(z)dz\circ \dot{W},
\end{equation}
where we use the fact that the exterior differential commutes with the stochastic differential and the spatial differential.

Then by wedging $dz$ on the above equation \eqref{variat}, we have
\begin{equation*}
dz\wedge \mathbf{M}dz_t+ dz\wedge \mathbf{K}d(\mathcal{G}z)=dz\wedge \left[\bigtriangledown_{zz}S_1(z)dz\right]
+dz\wedge \left[\bigtriangledown_{zz}S_2(z)dz\circ\dot{ W}\right].
\end{equation*}
Since $ \bigtriangledown_{zz}S_1(z)$ and $\bigtriangledown_{zz}S_2(z)$ are symmetric, $dz\wedge \bigtriangledown_{zz}S_1(z)dz=0$ and $dz\wedge \bigtriangledown_{zz}S_2(z)dz=0$ hold. Therefore we have
\begin{equation}
\label{combine1}
dz\wedge \mathbf{M}dz_t+ dz\wedge \mathbf{K}d(\mathcal{G}z)=0.
\end{equation}

Recalling that $\mathbf{M}$ and  $\mathbf{K}$ are skew-symmetric matrices, we have
\begin{equation*}
dz_t\wedge \mathbf{M}dz+d(\mathcal{G}z)\wedge \mathbf{K}dz=0.
\end{equation*}
Adding the above two equations yields
\begin{equation}
\label{combine2}
dz\wedge \mathbf{M}dz_t+dz_t\wedge \mathbf{M}dz+\left(dz\wedge \mathbf{K}d(\mathcal{G}z)+d(\mathcal{G}z)\wedge \mathbf{K}dz\right)=0.
\end{equation}
Note that
\begin{equation}
\label{wedge}
dz\wedge \mathbf{K}d(\mathcal{G}z)=d(\mathcal{G}z)\wedge \mathbf{K}dz.
\end{equation}
Then, we have
\begin{equation*}
d_t\Psi +\partial_{x}\kappa(x,t)=0, ~~a.s.
\end{equation*}

\end{proof}

\begin{remark}
Theorem \ref{msy} states that the stochastic nonlinear Schr\"odinger equation \eqref{orie07} exhibits a stochastic multi-symplectic conservation law, making it a stochastic multi-symplectic Hamiltonian system. When $\alpha=2$, this result reduces to the classical findings, as shown in \cite{CH16}. Furthermore, if we disregard the noise term, the system \eqref{symple} satisfies the generalized multi-symplectic conservation law
\begin{equation*}
\partial_t{\Psi}+\partial_{x}\kappa=0.
\end{equation*}
This outcome aligns with the deterministic results, as demonstrated in \cite[Theorem 2.3]{WH18}.
\end{remark}

The preservation of the stochastic symplectic structure has the following theorem. The proof is left in Appendix A.
\begin{lemma}(Preservation of the stochastic symplectic structure).
\label{presym}
The phase flow of stochastic nonlinear Schr\"odinger equation \eqref{orie07} preserves the stochastic symplectic structure.
\end{lemma}

\subsection{Structure-preserving numerical methods for SFNSE with the one-dimensional case}
Now, we will present a new numerical method, which preserves the discrete form of the stochastic multi-symplectic conservation law when it is applied to the stochastic multi-symplectic Hamiltonian system.

The SFNSE of interest has the form
\begin{equation}
\label{orie070}
\left\{
\begin{aligned}
&i du-\left[(-\Delta)^{\alpha} u+\lambda |u|^{2\sigma}u\right] dt=u \circ dW(t), \quad x \in (a,b), \quad t \geq 0, \\
&u(0)=u_0,
\end{aligned}
\right.
\end{equation}
Given a positive even integer $N$ with the mesh size $h=\frac{b-a}{N}$. Let $0=t_0<t_1<t_2<\cdots<t_{N}=T$ be a equidistant division of $[0,T]$ with time step $\delta t=\frac{T}{N}$ and $\triangle W_n=W(t_{n+1})-W(t_{n}),  n \in \{0,1,\cdots, N-1\},$ be the increments of the Brownian motion. Denote spatial grid points $x_j = a+ jh$ , for $0 \leq j \leq J$. Let $u^n_j$  be the numerical approximation of the
solution $u(x_j, t_n)$. Then the interpolation approximation $I_{N}u(x)$ of the function $u(x)$ has the follow form
\begin{equation}
I_{N}u(x)=\sum^{N/2}_{k=-N/2}\widetilde{u}_{k}e^{ik\mu(x-a)},
\end{equation}
where $\widetilde{u}_{k}=\frac{1}{Nc_k}\sum^{N-1}_{j=0}u(x_j)e^{-ik\mu(x_j-a)}$, $\mu=\frac{2\pi}{b-a}, c_k=1$ for $|k|<\frac{N}{2}$, and $c_k=2$ for $k=\pm \frac{N}{2}$.

Using the Fourier pseudospectral method \cite{SW11}, we can approximate the fractional Laplacian of $u(x_j,t)$ by
\begin{equation}
\label{0L}
-(-\Delta)^{\alpha}u_{N}(x_j,t)=-\sum^{N/2}_{k=-N/2}|k\mu|^{2\alpha}\widetilde{u}_{k}e^{ik\mu(x_j-a)},
\end{equation}
Plugging $\widetilde{u}_k$ into \eqref{0L}, we yields
\begin{equation}
\begin{aligned}
-(-\Delta)^{\alpha}u_{N}(x_j,t)&=-\sum_{l=0}^{N-1}u(x_l,t)\left(\sum^{N/2}_{k=-N/2}\frac{1}{Nc_k}|k\mu|^{2\alpha}\widetilde{u}_{k}e^{ik\mu(x_j-x_l)}\right):=(D^{\alpha}_2\mathbf{u})_{j},
\end{aligned}
\end{equation}
where $D^{\alpha}_2$ is an $N\times N$ matrix with elements
\begin{equation*}
(D^{\alpha}_2u)_{j,l}=\sum^{N/2}_{-N/2}\frac{1}{Nc_k}|k\mu|^{2\alpha}\widetilde{u}_ke^{ik\mu(x_j-x_l)},
\end{equation*}
and
\begin{equation*}
\mathbf{u}=(u_0, u_1, \cdots, u_{N-1})^{T}.
\end{equation*}
Similarly, the operator $\tilde{\mathcal{G}}$ can be approximated by
\begin{equation*}
\label{L}
\tilde{\mathcal{G}}u_{N}(x_j)=\sum^{N/2}_{-N/2}ik\mu|k\mu|^{\frac{\alpha-2}{2}}\widetilde{u}_{k}e^{ik\mu(x_j-a)}.
\end{equation*}
Plugging $\widetilde{u}_k$ into \eqref{L}, we yields

\begin{equation}
\begin{aligned}
\label{LL}
\tilde{\mathcal{G}}u_{N}(x_j,t)&=\sum_{l=0}^{N-1}u(x_l,t)\left(\sum^{N/2}_{k=-N/2}\frac{1}{Nc_k}ik\mu|k\mu|^{\alpha-1}e^{ik\mu(x_j-x_l)}\right):=(D^{\alpha}_1\mathbf{u})_j.
\end{aligned}
\end{equation}
where $D^{\alpha}_1$ is an $N\times N$ matrix with elements
\begin{equation*}
(D^{\alpha}_1u)_{j,l}=\sum^{N/2}_{k=-N/2}\frac{1}{Nc_k}ik\mu|k\mu|^{\frac{\alpha-2}{2}}e^{ik\mu(x_j-x_l)}.
\end{equation*}

Next we give the some elementary properties of matrices $D^{\alpha}_1$ and $D^{\alpha}_2$, which has been shown, see the reference \cite[Lemma 3.1]{WH18}.
\begin{lemma}
Matrix $D^{\alpha}_1$ and $\left(D^{\alpha}_1\right)^2$ are skew-symmetric matrices and $D^{\alpha}_2$ is a symmetric matrix, which the following relationship
\begin{equation}
\left(D^{\alpha}_2\right)_{j,l}=\left(D^{\alpha}_1\right)^2_{j,l}-(-1)^{j+l}\frac{1}{N}\left|\frac{N\mu}{2}\right|^{\alpha}
\end{equation}
holds.
\end{lemma}

\subsubsection{Stochastic multi-symplectic method}
We now discretize the system \eqref{seperate} in space by the Fourier pseudospectral method to obtain a semi-discrete system for $0\leq j \leq N-1$,

\begin{equation}
\label{0seperate0}
\left\{
\begin{aligned}
\frac{dp_j}{dt}+(D^{\alpha}_2\mathbf{q})_{j}-\lambda\left(p^2_{j}+q^2_{j}\right)^{\sigma}q_{j}-q_{j}\circ\dot{ W}_t=0, \\
\frac{dq_j}{dt}-(D^{\alpha}_2\mathbf{p})_{j}+\lambda\left(p^2_{j}+q^2_{j}\right)^{\sigma}p_{j}+p_{j}\circ\dot{ W}_t=0, \\
\end{aligned}
\right.
\end{equation}
where
\begin{equation*}
\mathbf{p}=(p_0, p_1, \cdots, p_{N-1})^{T}, ~~\mathbf{q}=(q_0, q_1, \cdots, q_{N-1})^{T}.
\end{equation*}

Now we discretize the system \eqref{0seperate0}  in the time midpoint rule to obtain  a semi-discrete scheme:

\begin{equation}
\label{0sham}
\left\{
\begin{aligned}
\frac{p^{n+1}_j-p^{n}_j}{\Delta t}+\left(D^{\alpha}_2\mathbf{q}^{n+\frac{1}{2}}\right)_{j}-\lambda\left[(p^{n+\frac{1}{2}})_j^2+(q^{n+\frac{1}{2}})_j^2\right]^{\sigma}q^{n+\frac{1}{2}}_j-q^{n+\frac{1}{2}}_{j}\triangle \dot{W}_n=0, \\
\frac{q^{n+1}_j-q^{n}_j}{\Delta t}-\left(D^{\alpha}_2\mathbf{p}^{n+\frac{1}{2}}\right)_{j}+\lambda\left[(p^{n+\frac{1}{2}})_j^2+(q^{n+\frac{1}{2}})_j^2\right]^{\sigma}p^{n+\frac{1}{2}}_j+p^{n+\frac{1}{2}}_{j}\triangle \dot{W}_n=0, \\
\end{aligned}
\right.
\end{equation}

where $p_{j+\frac{1}{2}}=\frac{1}{2}(p_j+p_{j+1}),~q_{j+\frac{1}{2}}=\frac{1}{2}(q_j+q_{j+1})$, and $\triangle W_n=W^{}(t_{n+1})-W(t_{n})$.\\

Let $\phi^{n}=p^n+\mathbf{i}q^n$, $\phi^{n+\frac{1}{2}}=\frac{\phi^{n}+\phi^{n+1}}{2}$, then the discrete system \eqref{0sham} can be rewritten as
\begin{equation}
\label{combination}
\frac{\mathbf{i}\left(\phi^{n+1}-\phi^{n}\right)}{\Delta t}+ D^{\alpha}_2\phi^{n+\frac{1}{2}}
-\lambda\left|\phi^{n+\frac{1}{2}}\right|^{2\sigma}\cdot \phi^{n+\frac{1}{2}}-\phi^{n+\frac{1}{2}}\triangle \dot{W}_n=0.
\end{equation}

The next lemma shows that the discrete system \eqref{combination}  preserves mass conservation.
\begin{lemma}
The discrete system \eqref{combination} is also  a stochastic Hamiltonian system, which satisfies the discrete mass conservation law, i.e.,
\begin{equation}
Q(\phi^{n+1})=Q(\phi^{n}), ~~0\leq n \leq N-1, ~~a.s.,
\end{equation}
where $Q(\phi^{n}):=\int_{\mathbb{R}}|\phi^{n}|^2dx $.
\end{lemma}
\begin{proof}
Multiplying \eqref{combination} by the complex conjugate $\frac{1}{2}(\phi^{n+1}+\phi^{n})$ on both sides, integrating in $\mathbb{R}$ and taking the imaginary part, we obtain
\begin{equation*}
\begin{aligned}
&\int_{\mathbb{R}}\left\{\mathbf{Im}\left(\frac{\mathbf{i}}{\Delta t}(\phi^{n+1}-\phi^{n})\cdot \frac{\overline{\phi^{n+1}}+\overline{\phi^{n}}}{2}\right)
+\mathbf{Im}\left(-\lambda\left|\phi^{n+\frac{1}{2}}\right|^{2\sigma} \phi^{n+\frac{1}{2}} \cdot \overline{\phi^{n+\frac{1}{2}}}\right)
+\mathbf{Im}\left(\phi^{n+\frac{1}{2}}\triangle \dot{W}_n \cdot \overline{\phi^{n+\frac{1}{2}}}\right)\right\}dx
\\
&=\int_{\mathbb{R}}\left\{\mathbf{Im}\left[\frac{\mathbf{i}}{2\Delta t}
\left(|\phi^{n+1}|^2-|\phi^{n}|^2+\overline{\phi^{n}}\phi^{n+1}-\phi^{n}\overline{\phi^{n+1}}\right)\right]
+\mathbf{Im}\left(-\lambda\left|\phi^{n+\frac{1}{2}}\right|^{2\sigma+2}\right)+\mathbf{Im}\left(\left|\phi^{n+\frac{1}{2}}\right|^2 \cdot \triangle \dot{W}_n\right)\right\}dx
\\&=\int_{\mathbb{R}}\left\{\frac{|\phi^{n+1}|^2-|\phi^{n}|^2}{2\Delta t}\right\}dx\\
&=0.
\end{aligned}
\end{equation*}
This implied that
\begin{equation*}
Q(\phi^{n+1})=Q(\phi^{n}), ~~0\leq n \leq N-1, ~~a.s.
\end{equation*}
\end{proof}

By using a similar technique as \cite[Theorem 3.1]{GM02}, we obtain that the discretized system \eqref{0sham} preserves the stochastic symplectic structure.
\begin{lemma}
The mid-point scheme preserves the stochastic symplectic structure, i.e.,
\begin{equation}
\sum_{j=0}^{N-1}dp^{n+1}_j\wedge dq^{n+1}_j=\sum_{j=0}^{N-1}dp^{n}_j\wedge dq^{n}_j.
\end{equation}
\end{lemma}

\subsection{Generalized stochastic multi-symplectic  method}

Applying the Fourier pseudospectral method to the multi-symplectic \eqref{seperate}, we obtain a semi-discrete system, for $0\leq j \leq N-1$,
\begin{equation}
\label{0seperate1}
\left\{
\begin{aligned}
\frac{dp_j}{dt}+(D^{\alpha}_1\mathbf{w})_{j}-\lambda\left(p^2_{j}+q^2_{j}\right)^{\sigma}q_{j}-q_{j}\circ\dot{ W}_t=0, \\
\frac{dq_j}{dt}-(D^{\alpha}_1\mathbf{v})_{j}+\lambda\left(p^2_{j}+q^2_{j}\right)^{\sigma}p_{j}+p_{j}\circ\dot{ W}_t=0, \\
(D^{\alpha}_1p)_{j}=v_j, \\
(D^{\alpha}_1q)_{j}=w_j,
\end{aligned}
\right.
\end{equation}
where $\mathbf{p}=(p_0, p_1, \cdots, p_{N-1})^{T},~~ \mathbf{q}=(q_0, q_1, \cdots, q_{N-1})^{T}, ~~\mathbf{v}=D^{\alpha}_1\mathbf{p}, ~~\mathbf{\omega}=D^{\alpha}_1\mathbf{q}$.

Moreover, we also have the following stochastic multi-symplectic conservation law for the semi-discrete system \eqref{0seperate1}. The proof is left in Appendix A.

\begin{lemma}
\label{corosym}
The system \eqref{0seperate1} satisfies the following semi-discrete stochastic multi-symplectic conservation law, i.e.,
\begin{equation}
\label{con2}
\frac{d}{dt}\Psi_j+\frac{1}{2}dz_j\wedge \mathbf{K}d\tilde{\mathcal{G}}z_j=0, ~~0\leq j\leq N-1, ~~a.s.,
\end{equation}
where
\begin{equation*}
z_j=(p_j, q_j, v_j, w_j)^{T}, ~~\Psi_j=\frac{1}{2}(dz_j\wedge \mathbf{M}dz_j), ~~\tilde{\mathcal{G}}dz_j=\sum_{k=0}^{N-1}(D^{\alpha}_1)_{j,k}dz_k.
\end{equation*}
\end{lemma}

\begin{remark}
Due to $D^{\alpha}_1 $ is a skew-symmetric matrix,  summing over $j$ yields $\frac{d}{dt}\sum^{N-1}_{j=0}\Psi_j=0$. This implies that the conservation of the global symplecticity over time.
\end{remark}

\subsection{Numerical example}
In this section, we consider the following stochastic fractional nonlinear Schr\"odinger equations to illustrate the efficiency of the splitting scheme and show the convergence order for the splitting scheme with spatially regular noise.
\begin{equation}
\label{orie00}
\left\{
\begin{aligned}
&i du-\left[(-\Delta)^{\alpha} u+\lambda|u|^{2\sigma}u\right] dt=\varepsilon u \circ dW(t), \quad x \in [-20,20], \quad t \geq 0, \\
&u(0)=u_0,
\end{aligned}
\right.
\end{equation}

Additionally, it's possible to acquire qualitative information about the impact of the noise. The numerical spatial domain spans $[-20,20]$, and the initial data is selected as $u(x,0)=\mathrm{sech}(x)\mathrm{exp}(2ix)$. We use a temporal step size of $\Delta t=0.01$, a spatial mesh grid size of $\Delta x=0.1$, and consider the time interval $[-20,20]$. We examine the real-valued Wiener process $W(t)=\sum_{l=1}^{100}\frac{1}{l}\sin(\pi lx)\beta_{l}(t)$, where $\{\beta_{l}(t)\}^{100}_{l=1}$ represents a set of independent $\mathbb{R}$-valued Wiener processes.

 \begin{figure}
\begin{center}
\includegraphics*[width=0.7\linewidth]{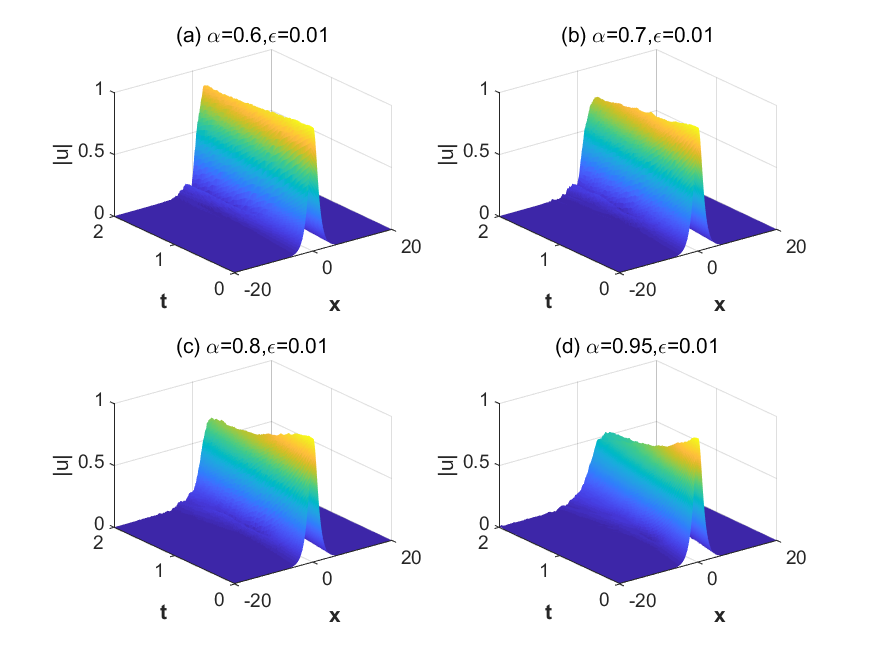}
\end{center}
\caption{The evolution of the solution $u(x,t)$ with the initial condition $u(x,0)=\mathrm{sech}(x)\mathrm{exp}(2ix)$, where $\epsilon=0.01$, $\sigma=1, \lambda=1$ and different $\alpha=0.6,0.7,0.8,0.95$. }
\label{orderError2}
\end{figure}

\begin{figure}
\begin{center}
\includegraphics*[width=0.7\linewidth]{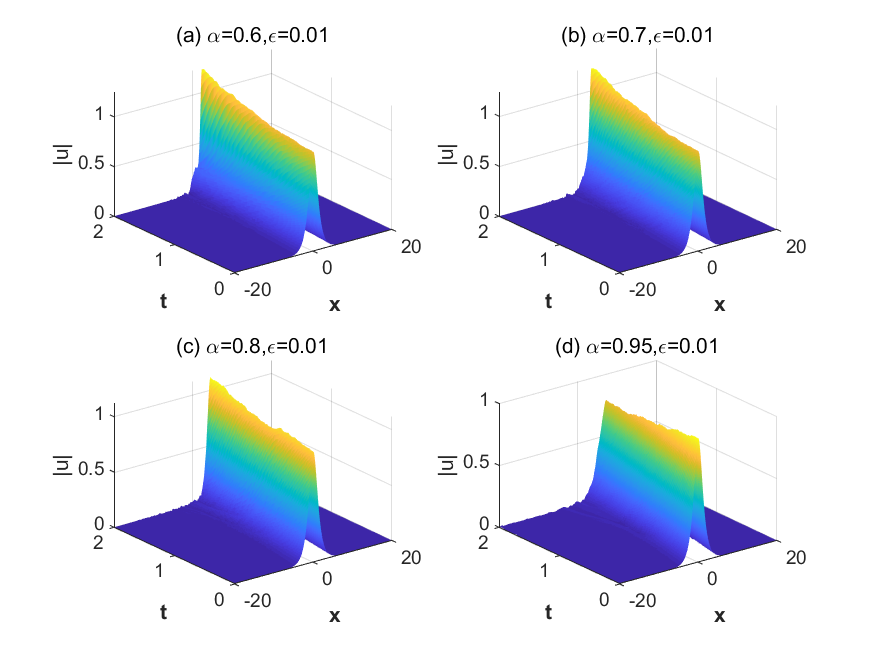}
\end{center}
\caption{The evolution of the solution $u(x,t)$ with the initial condition $u(x,0)=\mathrm{sech}(x)\mathrm{exp}(2ix)$, where $\epsilon=0.01$, $\sigma=1$, $\lambda=-1$ and different $\alpha=0.6,0.7,0.8,0.95$. }
\label{orderError22}
\end{figure}

 \begin{figure}
\begin{center}
\includegraphics*[width=0.7\linewidth]{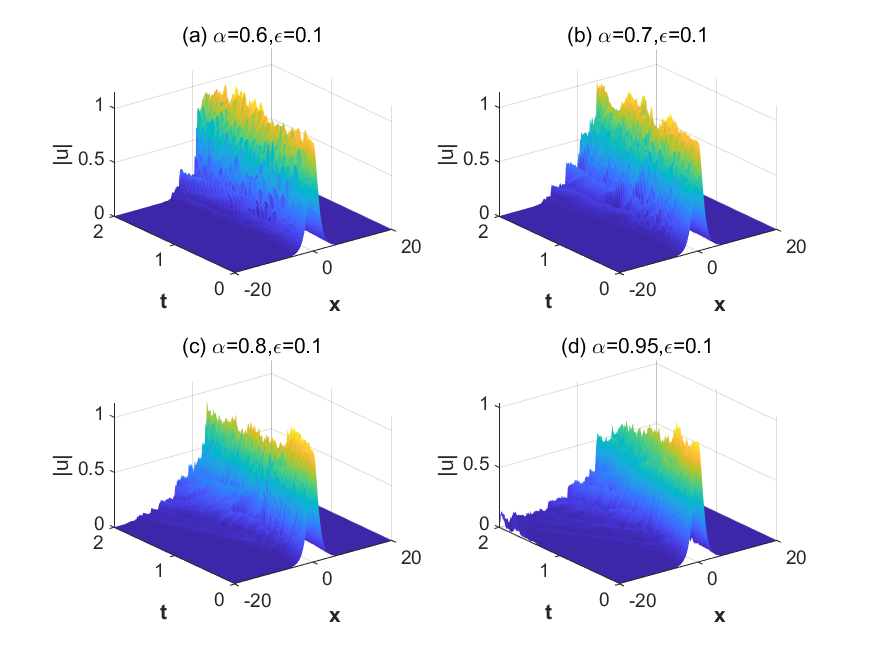}
\end{center}
\caption{The evolution of the solution $u(x,t)$ with the initial condition $u(x,0)=\mathrm{sech}(x)\mathrm{exp}(2ix)$, where $\epsilon=0.1$, $\sigma=1$, $\lambda=1$ and different $\alpha=0.6,0.7,0.8,0.95$. }
\label{orderError3}
\end{figure}

\begin{figure}
\begin{center}
\includegraphics*[width=0.7\linewidth]{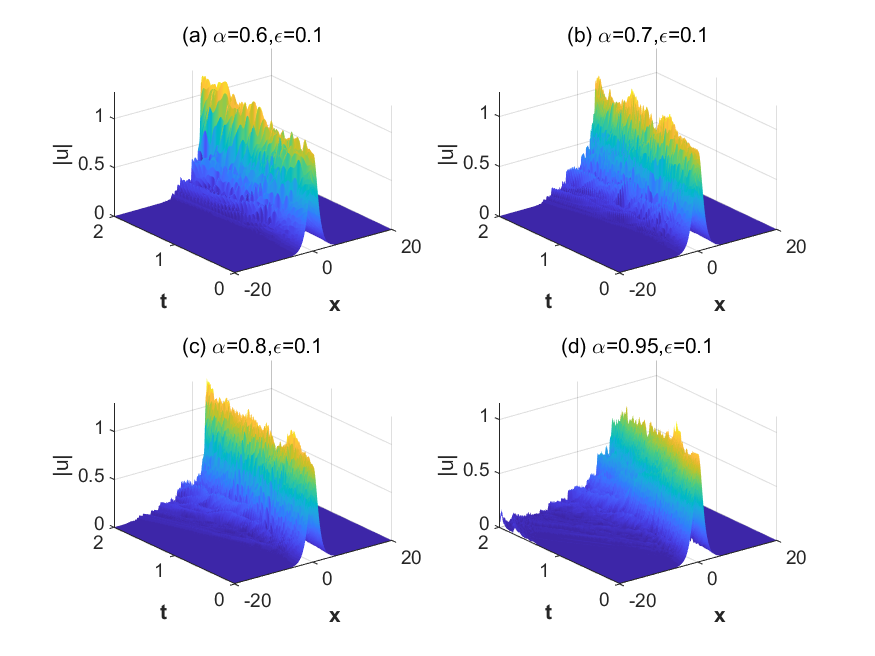}
\end{center}
\caption{The evolution of the solution $u(x,t)$ with the initial condition $u(x,0)=\mathrm{sech}(x)\mathrm{exp}(2ix)$, where $\epsilon=0.1$, $\sigma=1$, $\lambda=-1$ and different $\alpha=0.6,0.7,0.8,0.95$. }
\label{orderError33}
\end{figure}

\begin{figure}
\begin{center}
\includegraphics*[width=0.7\linewidth]{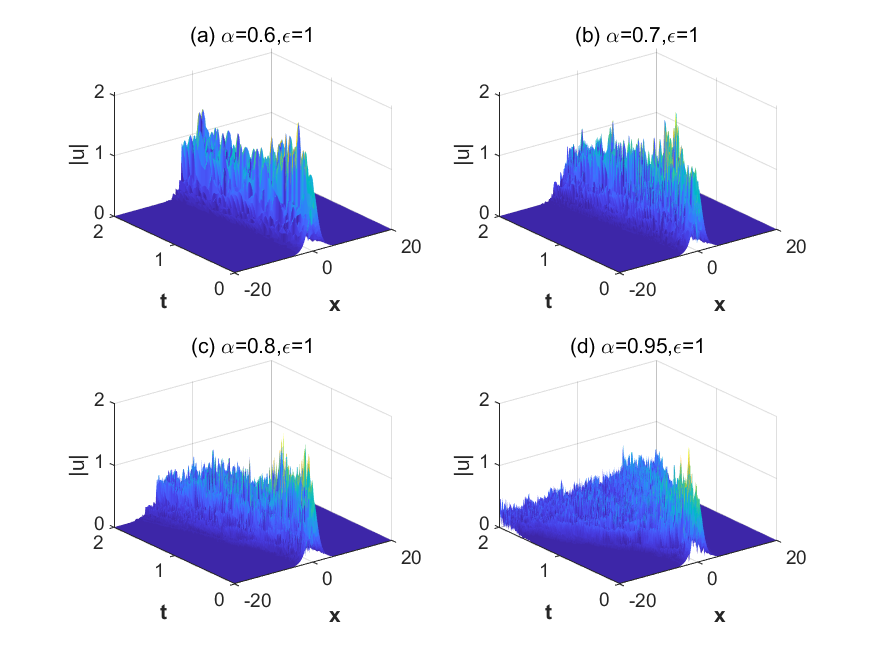}
\end{center}
\caption{The evolution of the solution $u(x,t)$ with the initial condition $u(x,0)=\mathrm{sech}(x)\mathrm{exp}(2ix)$, where $\epsilon=1$, $\sigma=1$, $\lambda=1$ and different $\alpha=0.6,0.7,0.8,0.95$. }
\label{orderError333}
\end{figure}

\begin{figure}
\begin{center}
\includegraphics*[width=0.7\linewidth]{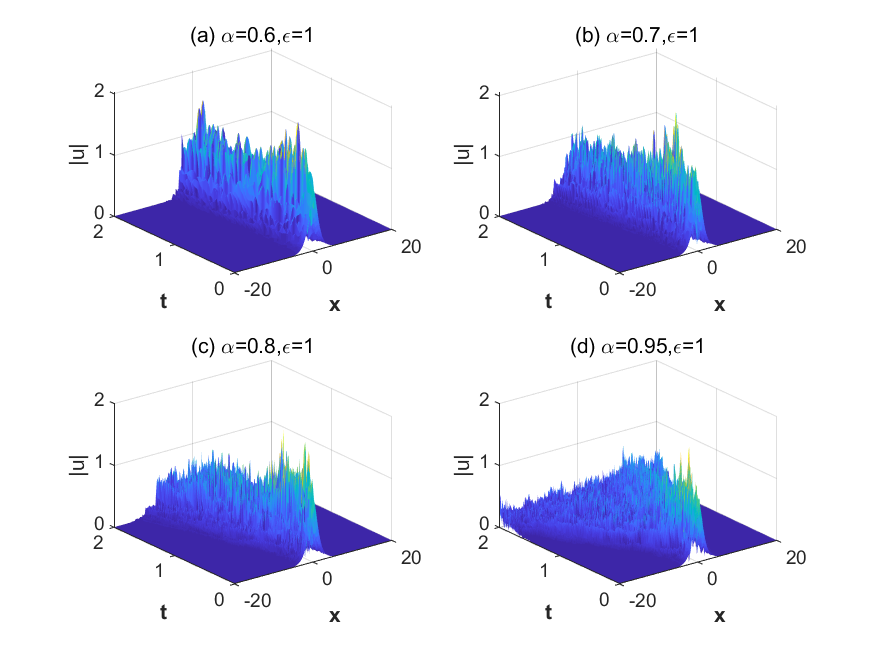}
\end{center}
\caption{The evolution of the solution $u(x,t)$ with the initial condition $u(x,0)=\mathrm{sech}(x)\mathrm{exp}(2ix)$, where $\epsilon=1$, $\sigma=1$, $\lambda=-1$ and different $\alpha=0.6,0.7,0.8,0.95$. }
\label{orderError}
\end{figure}
In Table 1, the global mass conservation is presented with different powers $\alpha=0.6,0.75,0.9$, $\epsilon=0.01$, $\sigma=1$, and various time values, denoted as $t$. The table illustrates that our numerical method maintains mass conservation.
\begin{table}[!hbp]
\centering
\begin{tabular}{cccc}

\hline
 Time & $\alpha=0.6$ &  $\alpha=0.75$   &  $\alpha=0.9$    \\
\Xhline{1.4pt}

 $t=0$  &   1.414211518677561 &  1.414211518677561 & 1.414211518677561  \\
 $t=2$  &    1.414211518677503 &  1.414211518677491 &   1.414211518677486  \\
 $t=4$  &    1.414211518677513 &  1.414211518677470 & 1.414211518677408 \\
 $t=6$  & 1.414211518677494 & 1.414211518677484 & 1.414211518677351\\
 $t=8$ &  1.414211518677447 & 1.414211518677443 &  1.414211518677257\\
 $t=10$ &  1.414211518677408 &  1.414211518677402 &   1.414211518677153 \\
\hline
\end{tabular}
\caption{The value of mass $M(t)$ at different time $t$ for stochastic nonlinear Schr\"odinger equation based on a single sample path .}
\label{MC}
\end{table}

 Figure 1~ and Figure 2~ depict the profile of the numerical solution $|u(x,t)|$ for a single trajectory with various powers $\alpha=0.6,0.7,0.8,0.95$, $\varepsilon=0.01$, and different $\lambda$. Figure 3 and Figure 4~ illustrate the profile of the numerical solution $|u(x,t)|$ for a single trajectory with different powers $\alpha=0.6,0.7,0.8,0.95$, $\varepsilon=0.1$, and different $\lambda$. Figure 5 and Figure 6~ show the profile of the numerical solution $|u(x,t)|$ for a single trajectory with different powers $\alpha=0.6,0.7,0.8,0.95$, $\varepsilon=1$, and different $\lambda$. From the above six figures, it is evident that minor noise does not disrupt the solitary wave and does not hinder its propagation, while high-level noise can affect the velocity of the solitary wave. Figure 7 shows the energy evolution with $\alpha=0.6$, $\lambda=-1$, and $\varepsilon=0.01$, by using the midpoint scheme and the same initial condition.  We can find that the energy is no longer constant in the stochastic case.

\begin{figure}
\begin{center}
\includegraphics*[width=0.7\linewidth]{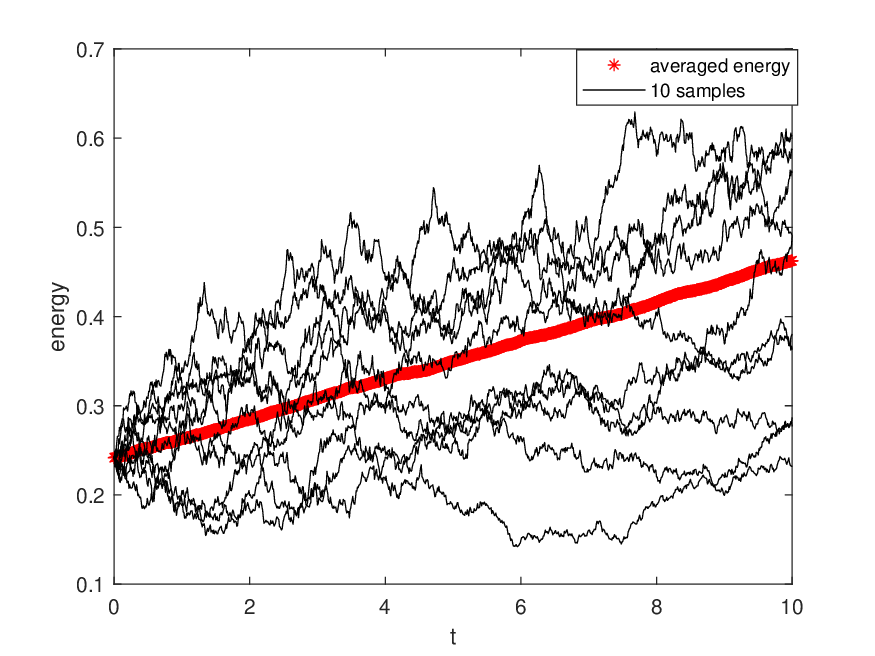}
\end{center}
\caption{The energy of 10 sample orbits and the averaged energy of 500 sample orbits. }
\label{energymean}
\end{figure}

\section{Concluding remarks.}
\par
In this paper, we first investigate the global existence of a solution for the stochastic fractional nonlinear Schr\"odinger equation with radially
symmetric initial data in a suitable energy space $H^{\alpha}$. We then show that the stochastic fractional nonlinear Schr\"odinger equation in the Stratonovich sense forms an infinite-dimensional stochastic Hamiltonian system, with its phase flow preserving symplecticity. Finally, we develop a stochastic midpoint scheme for the stochastic fractional nonlinear Schr\"odinger equation from the perspective of symplectic geometry. It is proved that the stochastic midpoint scheme satisfies the corresponding symplectic law in the discrete sense.

There are some limitations for this paper. Firstly, the condition $1<\alpha <2$ plays an important role in deriving the global well-posedness. How to obtain the global well-posedness and blow up criterion are still open problems with $ \alpha\in (0,1)$. Secondly, the midpoint scheme \eqref{0sham} is implicit, how to provide an explicit numerical scheme to solve the one-dimensional stochastic fractional nonlinear Schr\"odinger equation \eqref{orie07} with $\sigma\neq 0$, is also an active issue.

\section{Appendix A. Further Proofs.}\label{app-1}

\subsection{Proof of Lemma \ref{presym}.}

Spatially integrating the multi-symplectic conservation law \eqref{conservation}, we obtain
\begin{equation}
\label{equu}
\int_{\mathbb{R}^n}\Psi(t_1,x)-\Psi(t_0,x)dx=-\int^{t_1}_{t_0}\int_{\mathbb{R}^n}\partial_x\kappa(x,t)dx dt.
\end{equation}
The form of $\kappa$ yields that
\begin{equation*}
\partial_x\kappa(t,x)=-\frac{1}{2}dp\wedge d\mathcal{G}v-\frac{1}{2}dq\wedge d\mathcal{G} w.
\end{equation*}
Taking the above equality into \eqref{equu}, we obtain
\begin{equation*}
-\int^{t_1}_{t_0}\int_{\mathbb{R}^n}\partial_x\kappa(x,t)dxdt=\int^{t_1}_{t_0}\int_{\mathbb{R}^n}\left[\frac{1}{2}dp\wedge d\mathcal{G}v+\frac{1}{2}dq\wedge d\mathcal{G} w\right]dxdt.
\end{equation*}
The well-posedness of Eq.\eqref{orie07} yields that $\kappa(t, \cdot)$ vanishes. Thus we have
\begin{equation*}
\int_{\mathbb{R}^n}\Psi(t_1,x)dx=\int_{\mathbb{R}^n}\Psi(t_0,x)dx.
\end{equation*}
\subsection{Proof of Lemma \ref{corosym}.}

We rewrite the system \eqref{0seperate1} in a compact form
\begin{equation*}
\mathbf{M}\frac{dz_j}{dt}+\mathbf{K}\sum_{k=0}^{N-1}(D^{\alpha}_1)_{j,k}z_k=\nabla_{z}S_1(z_j)+\nabla S_2(z_j)\circ\dot{ W}_t,
\end{equation*}
and its variational equation is
\begin{equation*}
\mathbf{M}\frac{d}{dt}dz_j+\mathbf{K}\sum_{k=0}^{N-1}(D^{\alpha}_1)_{j,k}dz_k=\nabla_{zz}S_1(z_j)dz_j+\nabla_{zz} S_2(z_j)dz_j\circ\dot{ W}_t,
\end{equation*}
Taking the wedge product of the above equation with $dz_j$ and noting that
\begin{equation*}
dz_j\wedge \nabla_{zz}S_1(z_j)dz_j=0,
\end{equation*}
and
\begin{equation*}
dz_j\wedge \nabla_{zz}S_2(z_j)dz_j=0,
\end{equation*}
due to the symmetries of $\nabla_{zz}S_{1}(z_j)$ and $\nabla_{zz}S_{2}(z_j)$, and $\mathbf{M}, \mathbf{K}$ are skew-symmetric matrices, we obtain the multi-symplectic conservation law \eqref{con2}.

\section{Acknowledgments}
The authors would like to thank  Yichun Zhu (Chinese Academy of Sciences)  for helpful discussions.  The research of X. Wang
is supported by the Natural Science Foundation of Henan Province of China (Grant No. 232300420110).


\begin{thebibliography}{0}

\bibitem{Longhi15}
S. Longhi,
Fractional Schr\"odinger equation in optics,
Opt. Lett., 40: 1117-1120, 2015.






\bibitem{BM23}
S. Liu, Y. Zhang, B. A. Malomed, E. Karimi,
Experimental realizations of the fractional
Schr\"odinger equation in the temporal
domain, Nat. Commun., 14: 222, 2023.






\bibitem{Pusateri14}
A. D. Ionescu, F. Pusateri,
Nonlinear fractional Schr\"odinger equations in one dimension,
J. Funct. Anal., 266: 139-176, 2014.








\bibitem{Laskin02}
N. Laskin,
Fractional Schr\"odinger equation,
Phys. Rev. E,  66: 056108, 2002.


\bibitem{Kirkpatrick02}
K. Kirkpatrick, E. Lenzmann, Gigliola Staffilani,
On the continuum limit for discrete NLS with long-range lattice interactions,
Phys. Rev. E,  66: 056108, 2002.



\bibitem{IP14}
A. Ionescu, F. Pusateri,
Nonlinear fractional Schr\"odinger equations in one dimension,
J. Funct. Anal., 266: 139-176, 2014.



\bibitem{SZ}
X. Shang, J. Zhang,
Ground states for fractional Schr\"odinger equation with critical growth,
Nonlinearity,  27: 187-207, 2014.


\bibitem{CA23}
B. Choi, A. Ceves,
Continuum limit of $2D$ fractional nonlinear Schr\"odinger equation,
J. Evol. Equ., 23: 30, 2023.




\bibitem{LS15}
R. L. Frank, E. Lenzmann, L. Silvestre,
Uniqueness of radial solutions for the fractional Laplacian,
Commun. Pur. Appl. Math., 69, 1671-1726, 2016.


\bibitem{Huang15}
P. Wang, C. Huang,
An energy conservative difference scheme for the nonlinear fractional Schr\"odinger equations,
J. Comput. Phy., 293, 238-251, 2015.



\bibitem{DZ16}
S. Duo, Y. Zhang,
Mass-conservative Fourier spectral methods for solving the fractional nonlinear Schr\"odinger equation,
Comput. Math. with Appl., 71, 2257-2271, 2016.



\bibitem{Debussche99}
A. de Bouard,  A. Debussche,
A stochastic nonlinear Schr\"odinger equation with multiplicative noise,
Commun. Math. Phys.,  205: 121-127, 1999.

\bibitem{Debussche05}
A. de Bouard,  A. Debussche,
Blow-up for the stochastic nonlinear Schr\"odinger equation with multiplicative noise,
Ann. Probab., 33: 1078-1110, 2005.




\bibitem{rockner19}
S. Herr, M. R\"ockner, D. Zhang,
Scatter for stochastic nonlinear Schr\"odinger equations,
Commun. Math. Phys.,  368: 843-884, 2019.






\bibitem{Zhang17}
V. Barbu, M. R\"ockner, D. Zhang,
Stochastic nonlinear Schr\"odinger equations: No blow-up in the non-conservative case,
J. Differ. Equ.,  263: 7919-7940, 2017.




\bibitem{Deng22}
Y. Deng, A. R. Nahmod, H. Yue,
Random tensors, propagation of randomness, and nonlinear dispersive equations,
Invent. Math., 228: 539-686, 2022.




\bibitem{Deb02}
A. Debussche, L. D. Menza,
Numerical simulation of focusing stochastic nonlinear Schr\"odinger equations,
Physica D,  162: 131-154, 2002.





\bibitem{Deb04}
A. de Bouard,  A. Debussche,
A semi-discrete scheme for the stochastic nonlinear Schr\"odinger equation,
Numer. Math., 96: 733-770, 2004.







\bibitem{H19}
J. Cui, J. Hong, Z. Liu, W. Zhou,
Strong convergence rate of splitting schemes for stochastic nonlinear Schr\"odinger equations,
J. Differ. Equ., 266: 5625-5663, 2019.





\bibitem{JL130}
J. Liu,
Order of convergence of splitting schemes for both deterministic and stochastic nonlinear Schr\"odinger equations,
SIAM J. Numer. Anal., 51: 1911-1932, 2013.



\bibitem{JL13}
J. Liu,
A mass-preserving splitting scheme for the stochastic Schr\"odinger equation with multiplicative noise,
IMA J. Numer. Anal., 33: 1469-1479, 2013.




\bibitem{CH16}
C. Chen, J. Hong,
Symplectic Runge-Kutta semi-discretization for stochastic schr\"odinger equation,
SIAM J. Numer. Anal., 54: 2569-2593, 2016.

\bibitem{AC18}
Rikard Anton, David Cohen,
 Exponential integrators for stochastic Schr\"odinger equations driven by It\^o noise,
 J. Comput. Math., 36: 276-309, 2018.

\bibitem{CW17}
J. Cui, J. Hong, Z. Liu, W. Zhou,
Stochastic symplectic and multi-symplectic methods for nonlinear schr\"odinger equation with white noise dispersion,
J. Comput. Phy.,
342, 267-285, 2017.





\bibitem{CH18}
J. Cui, J. Hong,
Analysis of a splitting scheme for damped stochastic nonlinear Schr\"odinger equation with multiplicative noise,
SIAM J. Numer. Anal., 56: 2045-2069, 2018.

\bibitem{CHL17}
J. Cui, J. Hong, Z. Liu,
Strong convergence rate of finite difference approximations for stochastic cubic Schr\"odinger equations,
J. Differ. Equ., 263: 3687-3713, 2017.





\bibitem{CDC}
C.-Edouard Br\'ehier, D. Cohen,
Analysis of a splitting scheme for a class of nonlinear stochastic Schr\"odinger equations,
Appl. Numer. Math., 186: 57-83, 2023.








\bibitem{Chen17}
H. Yuan, G. Chen,
Martingale solutions of stochastic fractional nonlinear Schr\"odinger equation on a bounded interval,
J. Differ. Equ.,  263: 7919-7940, 2017.





\bibitem{ZZD}
A. Zhang, Y. Zhang, X. Wang, Z. Wang, J. Duan,
The stochastic fractional Strichartz estimate and blow-up for Schr\"odinger equation,
https://arxiv.org/abs/2308.10270.


\bibitem{Liu21}
Z. Brze\'zniak, W. Liu, J. Zhu,
The stochastic Strichartz estimates and stochastic nonlinear Schr\"odinger equations driven by L\'evy noise,
J. Funct. Anal., 281: 109021, 2021.






\bibitem{GM02}
G. Milstein, Yu. Repin, M. Tretyakov,
Symplectic integration of Hamiltonian systems with additive noise,
SIAM J. Numer. Anal., 40: 1583-1604, 2002.




\bibitem{WP19}
P. Wei, Y. Chao, J. Duan,
Hamiltonian systems with L\'evy noise: Symplecticity, Hamilton's principle and averaging principle,
Physica D, 292: 69-83, 2019.





\bibitem{Duan15}
J. Duan,
An Introduction to Stochastic Dynamics,
Cambridge University Press, 2015.






\bibitem{EG12}
E. Di Nezza, G. Palatucci, E. Valdinoci,
Hitchhiker's guide to the fractional Sobolev spaces,
Bull. Sci. Math., 136: 521-573, 2012.


\bibitem{VDD18}
V. D. Dinh,
A study on blowup solutions to the focusing $L^2$-supercritical nonlinear fractional Schr\"odinger equation,
J. Math. Phys., 59: 071506, 2018.

\bibitem{CW91}
M. Christ, I. Weinstein, Dispersion of small amplitude solutions of the generalized Korteweg-de Vries equation, J. Funct. Anal. 100: 87-109, 1991.




\bibitem{NN7}
N. N. Vakhania, V. I. Taireladze, S. A. Chobanyan,
Probability and distribution on Banach space,
Springer Science $\&$ Business Media, 1987.



\bibitem{CK93}
C. Kenig, G. Ponce, L. Vega,
Well-posedness and scattering results for the generalized Korteweg-de Vries equation
via contraction principle,
Commun. Pure Appl. Math., 46: 527-620, 1993.





\bibitem{JB14}
J. Bourgain, D. Li,
On an endpoint Kato-Ponce inequality,
Differ. Integral Equ., 27: 1037-1072, 2014.









\bibitem{DL19}
D. Li,
On Kato-Ponce and fractional Leibniz,
Rev. Mat. Iberoam, 35: 23-100, 2019.






\bibitem{AD03}
A. de Bouard, A. Debussche,
The stochastic non-linear Schr\"odinger equation in $H^1$,
Stoch. Anal. Appl., 21: 197-126, 2003.




\bibitem{ZD16}
V. Barbu, M. Roeckner, D. Zhang,
Stochastic nonlinear Schroedinger equations,
Nonlinear Anal., 136: 168-194, 2016.



\bibitem{VDD19}
V. Dinh,
On blow-up solutions to the focusing mass-critical nonlinear fractional Schr\"odinger equation,
Commun. Pur. Appl. Anal., 18: 689-708, 2019.



\bibitem{WH18}
P. Wang, C. Huang
Structure-preserving numerical method for the fractional Schr\"odinger equation,
Appl. Numer. Math., 129: 137-158, 2018.






\bibitem{SW11}
J. Shen, T. Tang, L. L. Wang,
Spectral Methods: Algorithm, Analysis and Applications, Springer Series in Computational Mathematics 41,
Springer, Berlin Heidelberg, 2011.













\bibitem{DC06}
D. Cohen,
Conservation properties of numerical integrators for highly oscillatory Hamiltonian systems.
IMA J. Numer. Anal.,
26: 34-59, 2006.




\bibitem{DC08}
D. Cohen,
Conservation of energy, momentum and actions in numerical discretizations of non-linear wave equation.
Numer. Math.,
110: 113-143, 2008.






\bibitem{GNM02}
G. N. Milstein, Yu. M. Repin and M. V. Tretyakov,
Symplectic integration of Hamiltonian systems with additive noise.
SIAM J. Numer. Anal., 39: 2066-2088, 2002.






\bibitem{CAA14}
C. A. Anton, Y. S. Wong and J. Deng,
Symplectic schemes for stochastic Hamiltonian systems preserving Hamiltonian functions.
 Int. J. Numer. Anal. Model., 11: 427-451, 2014.



%


\bibitem{KB12}
 K. Burrage and P. M. Burrage,
Low rank Runge-Kutta methods, symplecticity and stochastic Hamiltonian problems with additive noise.
J. Comput. Appl. Math., 236: 3920-3930, 2012.






\bibitem{JH19}
J. Hong, L. Miao and L. Zhang
Convergence analysis of a symplectic semi-discretization for stochastic NLS equation with quadratic potential.
Discrete Contin. Dyn. Syst.Ser. B, 24: 4295-4315, 2019.









\bibitem{GM95}
G. Milstein and M. Tretyakov,
Stochastic Numerics for Mathematical Physics,
Kluwer Academic Publishers, 1995.









\bibitem{TCA03}
T. Cazenave,
Semilinear Schrodinger Equations,
American Mathematical Society,
Courant Institute of Mathematical Sciences,  2003.




\bibitem{DZ92}
D. Prato and J. Zabczyk
Stochastic Equations in Infinite Dimensions.
Cambridge: Cambridge University Press, 1992.


\end{thebibliography}
\end{document}